\documentclass{amsart}   %%%
\usepackage{amssymb}\usepackage{amsmath}\usepackage{a4}
\def\dvol{\operatorname{dvol}}
\def\Id{\operatorname{Id}}
\def\Tr{\operatorname{Tr}}
\newcommand{\al}{{\vec\alpha}}
\newcommand{\D}{{\Delta_M}}
\newcommand{\nn}{\nonumber}
\newcommand{\vard}{y,r,\omega ,D_r,\lambda}
\newcommand{\var}{y,r,\omega,\tau,\lambda}

\newcommand{\gt}{\tilde g ^{ab}}
\newcommand{\gtr}{\tilde g ^{ab}_{,r}}
\newcommand{\gtra}{\tilde g ^{ab}_{,r}\omega_a \omega_b}

\newcommand{\gtrr}{\tilde g ^{ab}_{,rr}}

\newcommand{\tg}{\tilde g _{ab}}
\newcommand{\tgr}{\tilde g _{ab,r}}
\newcommand{\tgrr}{\tilde g _{ab,rr}}
\newcommand{\tl}{\tau^2 + \Lambda^2}
\newcommand{\tb}{\tilde b ^r}
\newcommand{\X}{\tilde g^{ab} \tilde g_{ab,rr}}
\newcommand{\Y}{\tilde g^{ab} \tilde g^{cd} \tilde g _{ac,r} \tilde g _{bd,r}}
\newcommand{\Z}{\tilde g^{ab} \tilde g^{cd} \tilde g_{ab,r} \tilde g _{cd,r}}
\def\iD{{\sqrt{-1}D}}
\def\BB{{\mathcal{B}}}\def\BBD{{\mathcal{B}_{\mathcal{D}}}}\def\BBa{{\mathcal{B}_a}}
\def\BBN{{\mathcal{B}_{\mathcal{N}}}}\def\BBR{{\mathcal{B}_{\mathcal{R}}}}
\def\DD{{\mathcal{D}}}\def\RR{{\mathcal{R}}}\def\NN{{\mathcal{N}}}

\begin{document}
\newtheorem{theorem}{Theorem}[section]
\newtheorem{lemma}[theorem]{Lemma}
\newtheorem{remark}[theorem]{Remark}
\newtheorem{example}[theorem]{Example}
\def\qedbox{\hbox{$\rlap{$\sqcap$}\sqcup$}}
\makeatletter
  \renewcommand{\theequation}{%
   \thesection.\alph{equation}}
  \@addtoreset{equation}{section}
 \makeatother
\def\Op{\operatorname{Op}}
\title[Heat trace asymptotics]
{Heat trace asymptotics with \\singular weight functions II}
\author{M. van den Berg, P. Gilkey, and K. Kirsten}
\begin{address}{MvdB: Department of Mathematics, University of
Bristol, University Walk, Bristol,\newline\phantom{...a}BS8 1TW, U.K.}\end{address}
\begin{email}{M.vandenBerg@bris.ac.uk}\end{email}
\begin{address}{PG: Mathematics Department, University of Oregon, Eugene, OR 97403, USA}\end{address}
\begin{email}{gilkey@uoregon.edu}\end{email}
\begin{address}{KK: Department of Mathematics, Baylor University \\ Waco, TX 76798, USA}\end{address}
\begin{email}{Klaus\_Kirsten@baylor.edu}\end{email}
\begin{abstract} We study the weighted heat trace asymptotics of an operator of Laplace type with mixed boundary conditions
where the weight function exhibits radial blowup. We give formulas for the first three boundary terms
in the expansion in terms of geometrical data.
\end{abstract}
\maketitle

\section{Introduction}\label{sect-1}
An important issue for several decades has been to obtain explicitly the coefficients of the short-time asymptotic expansion of the heat kernel associated with a Laplace type operator on a $m$-dimensional Riemannian manifold $M$ \cite{G94,K02}. In mathematics this interest stems in particular from the link between the spectrum of the operator and the underlying geometry of $M$ \cite{kac66-73-1}, but it extends to basically all of Geometric Analysis \cite{G94}. In physics the heat kernel asymptotic expansion has been realized to be a particularly useful tool to determine various approximations of effective actions and the Casimir energy \cite{ball86-263-245,byts96-266-1,eliz95b}.

Instead of simply analyzing the integrated heat trace one often puts a weight in the evaluation of the trace, sometimes called the localizing or smearing function. This function is introduced for various reasons. First it allows one to obtain local information from the integrated one, therefore, most importantly it is possible to recover the local behavior near the boundary. Furthermore, it is this smeared coefficient that appears in the integration of conformal anomalies relevant for several physical applications, see, e.g., \cite{blau89-4-1467,dowk89-327-267,K02}. For smooth localizing functions the results for the first few heat kernel coefficients are available for several years now \cite{gilk04b,K02}. A detailed analysis of what happens for singular weighing functions has only been started recently. In the context of the heat content asymptotics the weighing function plays the role of an initial temperature distribution. In the context of black hole physics singular conformal transformations play an important role when mapping black holes to their Penrose diagrams \cite{birr82b}.

The heat content asymptotics of an operator of Laplace type with singular initial temperature distribution and with Dirichlet or
Robin boundary conditions were investigated in \cite{BGS09}. A similar study of the heat
trace asymptotics with singular weighting function and Dirichlet boundary conditions was
performed in \cite{BGKS09}. In this paper, we conclude this line of investigation
by extending the results of \cite{BGKS09} concerning heat trace asymptotics to
Robin, and more generally, to mixed boundary conditions. We anticipate that also this singular setting will find its applications in
physics.

\subsection{Operators of Laplace type}\label{sect-1.1}
Let $M$ be a compact Riemannian manifold of dimension $m$ with smooth non-empty boundary $\partial M$. Let
$V$ be a smooth vector bundle over $M$ and let $D$ be an operator of Laplace type on the space of smooth sections $C^\infty(V)$.
This means that locally we may express $D$ in the form
\begin{equation}\label{eqn-1.1}D=-(g^{\mu\nu}\partial_{x_\mu}\partial_{x_\nu}\Id+A^\nu\partial_{x_\nu}+B)\end{equation}
for suitably chosen matrices $A^\nu$ and $B$ where we adopt the {\it Einstein convention} and sum over repeated indices and where
$g^{\mu\nu}$ denotes the inverse matrix. It is possible to express $D$ invariantly
\cite{G94} using a Bochner formalism. There exists a unique connection $\nabla$ on $V$ and
a unique endomorphism $E$ of $V$ so that
$$D\phi=-(g^{\mu\nu}\phi_{;\mu\nu}+E\phi)\, ,$$
where we use `;' to denote the components of multiple covariant differentiation. Let $\Gamma$ be the Christoffel symbol of the
Levi-Civita connection. We then have
\begin{equation}\label{eqn-1.a}
\begin{array}{l}
\omega_\delta=\textstyle\frac12g_{\nu\delta}(A^\nu+g^{\mu\sigma}\Gamma_{\mu\sigma}{}^\nu
\operatorname{Id}),\\
E=B-g^{\nu\mu}(\partial_{x_\mu}\omega_{\nu}+\omega_{\nu}\omega_{\mu}
    -\omega_{\sigma}\Gamma_{\nu\mu}{}^\sigma)\,.\vphantom{\vrule height 11pt}
\end{array}\end{equation}

\subsection{Boundary conditions}\label{sect-1.2}
We recall the formalism of Branson and Gilkey
\cite{BG92}. Let $\varepsilon>0$ be the injectivity radius of the boundary $\partial M$ in
$M$. Use the geodesic flow defined by the unit inward normal vector field $\partial_r$ to
define a diffeomorphism between the collar
${\mathcal{C}_{\varepsilon}}:=\partial M\times[0,\varepsilon]$ and a neighborhood of the boundary in $M$ which identifies
$\partial M\times\{0\}$ with $\partial M$. The curves
$r\rightarrow (y_0,r)$ for $r\in[0,\varepsilon]$ are unit speed geodesics perpendicular to the
boundary and $r$ is the geodesic distance to the boundary.

Let  $\chi\in C^\infty(\operatorname{End}(V|_{\partial M}))$ satisfy
$\chi^2=1$. Extend $\chi$ to the collar $\mathcal{C}_\varepsilon$ so that $\nabla_{\partial_r}\chi=0$. Let
$\Pi_\pm:=\frac12(1\pm\chi)$ be projections on the $\pm1$ eigenbundles $V_\pm$ of $\chi$. Let
$S\in\operatorname{End}(V|_{\partial M})$ be an auxiliary endomorphism with
$\Pi_+S=S\Pi_+=S$.
If
$\phi\in C^\infty(V)$, let $\BB=\BB(\chi,S)$ be the {\it mixed boundary operator}:
\begin{equation}\label{eqn-1.b}
\BB\phi:=\left.\left\{\Pi_-\phi\right\}\right|_{\partial M}
    \oplus\left.\left\{\Pi_+(\nabla_{\partial_r}+S)\Pi_+\phi\right\}\right| _{\partial M}\,.
\end{equation}
Let $D_\BB$ be the realization of $D$ with this boundary condition. We set $\Pi_+=0$ to define the Dirichlet
boundary operator
$\BBD$ and we set $\Pi_-=0$ and $S=0$ to define the Neumann boundary operator $\BBN$.

Operators of this type arise when studying
the Gauss-Bonnet theorem for manifolds with boundary \cite{G94} and will play an important role in
the analysis of Section \ref{sect-4}. Let $\Delta=d\delta+\delta d$ be the Laplace-Beltrami operator on the space of smooth differential forms.
Let $(y,r)$ be coordinates on the collar $\mathcal{C}_\epsilon$. Set
$$\Pi_+(dy_I)=dy_I\quad\text{and}\quad\Pi_-(dy_I\wedge dr)=dy_I\wedge dr$$
and define the {\it absolute boundary operator} $\BBa$ by taking
$$\BBa\{\phi_Idy^I+\psi_Jdy^J\wedge dr\}=\{(\partial_r\phi_I)dy^I\}|_{\partial M}\oplus\{\psi_Jdy^J\}|_{\partial M}\,.$$
Extend the second fundamental form $L$ (see Section \ref{sect-1.6} below) to act as a derivation on the space of differential forms.
Then:
$$\nabla_{\partial r}(f_Idy^I)=(\partial_r+L)(f_Idy^I)\quad\text{so}\quad\BB_a=\BB(\chi,-L)\,.$$
Let $\Delta_p$ be Laplacian on the space of smooth $p$-forms. If $M$ is a closed manifold, then
$\ker(\Delta_{p,\BB_a})$ is naturally isomorphic to the topological cohomology groups $H^p(M;\mathbb{C})$. Relative
boundary conditions $\BB_r$ are defined similarly using the Hodge $\star$ operator and one may identify
$\ker(\Delta_{p,\BB_a})$ with the relative cohomology groups $H^p(M,\partial M;\mathbb{C})$.

\subsection{The heat equation}\label{sect-1.3}
For $t>0$ and $\phi\in L^2(V)$, let $u=e^{-tD_\BB}\phi$ be the solution of the heat equation:
$$(\partial_t+D_\BB)u(x;t)=0,\quad \BB u=0,\quad \lim_{t\downarrow0}u(\cdot;t)=\phi(\cdot)\text{ in }L^2(V)\,.$$
Let $\dvol_M$ (resp. $\dvol_{\partial M}$) be the Riemannian measure on $M$ (resp. $\partial M$). There is a
smooth kernel
$p_{D,\BB}(x,\tilde x;t)$ which gives the fundamental solution of the heat equation:
$$u(x;t)=\int_Mp_{D_\BB}(x,\tilde x;t)\phi(\tilde x)\dvol_M(\tilde x)\,.$$
If $D_\BB$ is formally self-adjoint with respect to a fiber metric, we can take a spectral resolution
$\{\lambda_\nu,\theta_\nu\}$ for
$D_\BB$ where
$\{\theta_\nu\}$ is a complete orthonormal basis for $L^2$ with $\BB\theta_\nu=0$ and $D\theta_\nu=\lambda_\nu\theta_\nu$.
We then have
$$p_{D_\BB}(x,\tilde x,t)=\sum_\nu e^{-t\lambda_\nu}\theta_\nu(x)\theta_\nu(\tilde x)\,.$$
This series converges in the $C^\infty$ topology for $t>0$. (There are some additional notational complexities in the bundle
valued case we suppress in the interests of simplicity).

\subsection{Weighting functions}\label{sect-1.4}
We study the weighted heat trace $\Tr_{L^2}(Fe^{-tD_\BB})$. Previous work has concentrated on the smooth section - we review
that work presently in Section \ref{sect-1.5}. However, in this paper, we shall concentrate on a more general setting
and consider the following class of smearing or weighting functions. Let $\alpha<1$. Let $F$ be a smooth
function on the interior of $M$. We assume that $r^\alpha F$ is smooth on the collar
$\mathcal{C}_\varepsilon :=\partial M \times [0,\varepsilon ]$; the parameter $\alpha$ controls the growth
(if $\alpha>0$) or decay (if $\alpha<0$) of $F$ near the boundary. We expand $F$ in a modified Taylor series
near the boundary:
\begin{eqnarray*}
&&F(y,r)\sim r^{-\alpha} (F_0(y)+rF_1(y)+r^2F_2(y)+....)\quad\text{where}\\
&&\textstyle F_i(y)=\frac1{i!}(\partial_{r})^i\{r^\alpha F\}|_{r=0}\,.
\end{eqnarray*}
We remark that the assumption that $\alpha<1$ ensures that $F\in L^1(M)$. With Dirichlet boundary conditions, the fundamental solution of the heat equation vanished to second
order on the boundary and it was possible to consider the region
$\alpha<3$; logarithmic singularities then appeared when $\alpha=1,2$. This is not possible in the more
general situation since the fundamental solution of the heat equation
$p_{D_\BB}$ need not vanish on $\partial M$ and we must restrict to $\alpha<1$ to ensure convergence.

\subsection{Heat trace asymptotics in the smooth setting}\label{sect-1.5}
Suppose $\alpha=0$ so that $F$ is smooth on all of $M$; this is the case considered classically. Work of Greiner
\cite{Gr68} and of Seeley
\cite{Se69} shows:
\begin{theorem}\label{thm-1.1}
Let $D$ be an operator of Laplace type on a compact Riemannian manifold $M$ with smooth boundary. Let $D_\BB$ be the
realization of $D$ with respect to the mixed boundary conditions
$\BB$ given in {\rm Equation (\ref{eqn-1.b})}. There is a full asymptotic series as
$t\downarrow0$ of the form:
\begin{eqnarray*}
&&\Tr_{L^2}(Fe^{-tD_\BB})\sim(4\pi)^{-m/2}t^{-m/2}\sum_{n=0}^\infty
t^na_n(F,D)\\
&&\qquad\qquad\quad\qquad+(4\pi)^{-m/2}t^{-(m-1)/2}\sum_{\ell=0}^\infty
t^{\ell/2}a_\ell^{bd}(F,D,\BB)\,.
\end{eqnarray*}
There are local invariants defined on $M$ and on $\partial M$ so that
\begin{eqnarray*}
&&a_n(F,D)=\int_MF(x)a_n(x,D)\dvol_M(x),\\
&&a_\ell^{bd}(F,D,\BB)=\sum_{i=0}^\ell\int_{\partial M}F_i (y) a_{\ell,i}^{bd}(y,D,\BB)\dvol_{\partial M}(y)\,.
\end{eqnarray*}
\end{theorem}
These invariants play an important role in index theory; they are also important in regularization results for mathematical
physics \cite{G94,K02}. We remark that we have used a different indexing convention than
is sometimes used in the literature and that we have handled the normalizing constants involving
$4\pi$ slightly differently.

\subsection{Local formulas}\label{sect-1.6}
One has explicit combinatorial formulas for these invariants in the smooth setting; the interior
invariants are known for $n\le5$ \cite{avra91-355-712,full88-310-583,ven98-15-2311} and the boundary invariants are known for
$\ell\le5$ \cite{bran99-563-603,gilk04b,K02,vass03-388-279}. We introduce the requisite notation as follows.

Let $R_{ijkl}$ be the components of the curvature tensor of the Riemannian manifold; with our sign convention, $R_{1221}=+1$
on the unit sphere in $\mathbb{R}^2$. Let
$\tau:=R_{ijji}$ be the scalar curvature of the manifold.  Let $\rho_{mm}=R_{immi}$ be the normal
component of the Ricci tensor. Let $L_{ab}$ be the components of the second fundamental form on the
boundary relative to an orthonormal frame
$\{e_1,...,e_{m-1}\}$ for the tangent bundle of $\partial M$;
$L_{ab}=g(e_m,\nabla_{e_a}e_b)$. Relative to the coordinate frame, we have
$$L(\partial_\mu,\partial_\nu)=\Gamma_{\mu\nu m}=-\textstyle\frac12\partial_rg_{\mu\nu}\quad\text{for}\quad1\le\mu,\nu\le
m-1\,.$$
Thus $L_{11}=1$ for the unit disk in $\mathbb{R}^2$. Express $D=-(g^{\mu\nu}\nabla_\nu\nabla_\mu+E)$ where $\nabla$ and $E$ are
as in Equation (\ref{eqn-1.a}). Then:
\begin{theorem}\label{thm-1.2}
\ \begin{enumerate}
\item $a_0(x,F,D)=\Tr\{F\Id\}$.
\smallbreak\item
$a_1(x,F,D)=\Tr\{6FE+F\tau\Id\}$.
\end{enumerate}\end{theorem}

\begin{theorem}\label{thm-1.3}
\ \begin{enumerate}
\item $a_0^{bd} (y,F,D,\BB)=\frac{\sqrt{4\pi}}4\Tr\{F_0\chi\}$,
\smallbreak\item
$a_1^{bd}(y,F,D,\BB)=\frac{\sqrt{4\pi}}6\Tr\{2F_0L_{aa}\Id+3F_1\chi+12F_0S\}$,
\smallbreak\item
$a_2^{bd}(y,F,D,\BB)=\frac{\sqrt{4\pi}}{384}\Tr\{F_0[96\chi
E+16\tau\chi-8\rho_{mm}\chi+L_{aa}L_{bb}(13\Pi_+-7\Pi_-)$
\smallbreak$
+L_{ab}L_{ab}(2\Pi_++10\Pi_-)+96L_{aa}S
+192S^2-12\chi_{;a}\chi_{;a}]$\smallbreak$+F_1[L_{aa}(6\Pi_++30\Pi_-)+96S]+48F_2\chi\}$.
\end{enumerate}
\end{theorem}

\subsection{The shifted asymptotic series}\label{sect-1.7}
If $\alpha\ne0$, there is a shift in the power of $t$ for the boundary invariants but the interior series
discussed in Section \ref{sect-1.5} is unchanged. In
\cite{BGKS09}, we used the calculus of pseudo-differential operators to establish the existence of an asymptotic series with
Dirichlet boundary conditions. The same approach extends directly to the situation at hand to yield the following generalization
of Theorem \ref{thm-1.1}:
\begin{theorem}\label{thm-1.4}
Let $D$ be an operator of Laplace type on a compact Riemannian manifold $M$ with smooth boundary. Let $D_\BB$ be the
realization of $D$ with respect to the mixed boundary conditions
$\BB$ given in {\rm Equation (\ref{eqn-1.b})}. Let
$\alpha<1$. Let
$F$ be smooth on the interior of
$M$ and let $r^\alpha F$ be smooth near the boundary. There is a full asymptotic series as
$t\downarrow0$ of the form:
\begin{eqnarray*}
&&\Tr_{L^2}(Fe^{-tD_\BB})
\sim(4\pi)^{-m/2}t^{-m/2}\sum_{n=0}^\infty t^na_n(F,D)\\
&&\qquad\qquad\qquad\quad+(4\pi)^{-m/2}t^{-(m-1)/2}\sum_{\ell=0}^\infty
t^{(\ell-\alpha)/2}a_{\ell,\alpha}^{bd}(F,D,\BB)\,.
\end{eqnarray*}
The interior invariants $a_n(F,D)$ are as discussed in Theorem \ref{thm-1.1}. There are local invariants
$a_{\ell,\alpha,i}^{bd}(y,D,\BB)$ which are real analytic in the parameter $\alpha$, so
$$a_{\ell,\alpha}^{bd}(F,D,\BB)=
\int_{\partial M}\sum_{i=1}^\ell F_i(y)a_{\ell,\alpha,i}^{bd}(y,D,\BB)\dvol_{\partial M}(y)\,.$$
\end{theorem}

The interior invariants do not depend either upon $\alpha$ or upon $\BB$ and are described for $n=0$ and $n=1$ by Theorem
\ref{thm-1.2} (and are known explicitly in the literature for $n\le5$); thus our attention will be concentrated on the boundary
invariants and upon extending Theorem
\ref{thm-1.3} to this more general setting. The analyticity of the
invariants in $\alpha$ will play a crucial role. We shall often restrict to the case
$\alpha\notin\mathbb{Z}$ in proving certain identities to ensure that the interior terms and the boundary terms do not interact.
We shall also often assume $\alpha<<0$ to avoid convergence questions. The result for general
$\alpha$ will then follow by analytic continuation since the local invariants are real analytic in the parameter $\alpha$. We
set the local boundary heat trace density to be:
$$a_{\ell,\alpha}^{bd}(y,F,D,\BB):=\sum_{i=1}^\ell F_i(y)a_{\ell,\alpha,i}^{bd}(y,D,\BB)$$
\subsection{Dirichlet Boundary Conditions}
We computed the boundary invariants for Dirichlet boundary conditions in \cite{BGKS09};
the following result is a consequence of those computations and forms an essential starting point for the study of the general
case:

\begin{theorem}\label{thm-1.5}
Let $\BB$ define Dirichlet boundary conditions. Let
$\textstyle\kappa_\alpha:=\frac12\Gamma\left(\frac{1-\alpha}2\right)$.
\begin{enumerate}
\item $a_{0,\alpha}^{bd}(y,F,D,\BB)=\kappa_\alpha
\Tr\{-F_0\Id\}$.
\smallbreak\item
$a_{1,\alpha}^{bd}(y,F,D,\BB)=\kappa_{\alpha-1}
\Tr\{-F_1\Id+\frac{\alpha-4}{2(\alpha-3)}F_0L_{aa}\Id\}$.
\smallbreak\item $a_{2,\alpha}^{bd}(y,F,D,\BB)=\kappa_{\alpha-2}
\Tr\{-F_2\Id+\frac{\alpha-5}{2(\alpha-4)}F_1L_{aa}\Id
$\smallbreak\quad$
+\frac{1}{6}F_0R_{amma}\Id
-\frac{\alpha-7}{8(\alpha-6)}F_0L_{aa}L_{bb}\Id
+\frac{\alpha-5}{4(\alpha-6)}F_0L_{ab}L_{ab}\Id$\smallbreak\quad$
-\frac1{3(1-\alpha)}F_0R_{ijji}\Id-\frac2{1-\alpha}F_0E\}$.
\end{enumerate}
\end{theorem}

\subsection{Heat trace asymptotics for mixed boundary invariants}
The following is the main result of this paper. It generalizes Theorem \ref{thm-1.5} to general mixed boundary conditions:

\begin{theorem}\label{thm-1.6}
Let $a_{\ell,\alpha}$ be the invariants of {\rm Theorem \ref{thm-1.4}}.
\begin{enumerate}
\item $a_{0,\alpha}^{bd}(y,F,D,\BB)=\kappa_{\alpha}\Tr\{F_0[\Pi_+-\Pi_-]\}$.
\smallbreak\item $a_{1,\alpha}^{bd}(y,F,D,\BB)=\kappa_{\alpha-1}
\Tr\{F_1[\Pi_+-\Pi_-]$
\smallbreak
$+F_0L_{aa}[\frac{\alpha^2-\alpha-4}{2(\alpha-1)(3-\alpha)}\Pi_++\frac{\alpha-4}{2(\alpha-3)}\Pi_-]+\frac4{1-\alpha}F_0S\}$.
\smallbreak\item
$a_{2,\alpha}^{bd}(y,F,D,\BB)=k_{\alpha-2}\Tr\{F_2[\Pi_+-\Pi_-]$
\smallbreak$
+F_1[L_{aa}(\frac{\alpha^2-3\alpha-2}{2(\alpha-2)(4-\alpha)}\Pi_++\frac{\alpha-5}{2(\alpha-4)}\Pi_-)
    +\frac4{2-\alpha}S]$
\smallbreak$
+F_0[-\frac16\rho_{mm}(\Pi_+-\Pi_-)$
\smallbreak\qquad$
+L_{aa}L_{bb}(\frac{ \alpha ^4 - 6 \alpha ^3 - \alpha ^2 -2\alpha + 104}{8 (\alpha -6) (\alpha -4)
(\alpha -2) (\alpha -1)}\Pi_+ -\frac{(\alpha-7)}{8(\alpha-6)}\Pi_-)$
\smallbreak\qquad$
+L_{ab}L_{ab}[ - \frac{ \alpha ^3 - 10 \alpha ^2 + 21 \alpha +4}{4 (\alpha -6) (\alpha -4) (\alpha
-1)}\Pi_++\frac{\alpha-5}{4(\alpha-6)}\Pi_-]$
\smallbreak\qquad$
+ \frac{ 2 (\alpha ^2-\alpha -8)} {(\alpha -1) (\alpha -2) (\alpha
-4)}L_{aa}S+\frac8{(2-\alpha)(1-\alpha)}S^2$
\smallbreak\qquad$+
\frac1{3(1-\alpha)}(\tau+6E)(\Pi_+-\Pi_-)- \frac 1 {(\alpha -1) (\alpha -4)}\chi_{;a}\chi_{;a}]\}$.
\end{enumerate}
\end{theorem}

Here is a brief outline to the paper. In Section \ref{sect-2}, we express the invariants $a_{0,\alpha}^{bd}$, $a_{1,\alpha}^{bd}$,
and $a_{2,\alpha}^{bd}$ in terms of geometrical quantities with 8 undetermined universal coefficients $\vartheta_\alpha^i$; we
refer to Lemma \ref{lem-2.8} for details. In Section \ref{sect-3}, we
determine the coefficient of $S$ in $a_1$ and the coefficient of $S^2$ in $a_2$ by performing a computation on the interval.
In Section \ref{sect-4}, we examine absolute and relative boundary conditions in dimension 2 to derive additional
relations. In Section \ref{sect-5}, we use the calculus of pseudo-differential operators to
complete the computation.

\section{The method of universal coefficients}\label{sect-2}
\subsection{Weighted homogeneity and dimensional analysis} We assign {\it weight} $k$ to the $k^{\operatorname{th}}$ derivative
of the metric, weight
$k+1$ to the $k^{\operatorname{th}}$ derivative of the connection form $\omega$ of Equation (\ref{eqn-1.a}), and weight $k+2$ to
the $k^{\operatorname{th}}$ derivative of the endomorphism $E$ of Equation (\ref{eqn-1.a}). We also assign weight $k$ to the
$k^{\operatorname{th}}$ tangential derivative of $\chi$, weight $k$ to $F_k$, and weight $k+1$ to the $k^{\operatorname{th}}$
tangential derivative of $S$. Thus, in particular, the components $R_{ijkl}$ of the curvature tensor have weight $2$ and the
components
$L_{ab}$ of the second fundamental form have weight $1$. Standard arguments using dimensional analysis shows establishes the
following result; we omit details in the interests of brevity and instead refer to \cite{BGKS09,G94,K02} where similar results
were established:
\begin{lemma}\label{lem-2.1}
The local invariants $a_{l,\alpha,i}^{bd}$ of Theorem \ref{thm-1.4} are weighted homogeneous of degree $\ell-i$.
\end{lemma}

\subsection{Orthogonal invariants} Weyl's theory of orthogonal invariants \cite{W46} may be used to construct a spanning set
for the space of invariants which are homogeneous of weight $k$. One uses the metric to contract indices in pairs. We let
$\chi_{:a}$ denote the components of tangential covariant differentiation of the tensor $\chi$. Lemma \ref{lem-2.2} then leads
to the following result; again, we omit details as by now the arguments are standard:

\begin{lemma}\label{lem-2.2}
There exist
universal constants
$\{\varrho_\alpha^{i,\pm}\}$ so that:
\begin{enumerate}
\item $a_{0,\alpha}^{bd}(y,F,D,\BB)=\Tr\{F_0[\varrho_\alpha^{0,+}\Pi_++
\varrho_\alpha^{0,-}\Pi_-]\}$.
\smallbreak\item $a_{1,\alpha}^{bd}(y,F,D,\BB)=
\Tr\{F_1[\varrho_{\alpha}^{1,+}\Pi_++\varrho_{\alpha}^{1,-}\Pi_-]$
\smallbreak
$+F_0[L_{aa}(\varrho_\alpha^{2,+}\Pi_++\varrho_\alpha^{2,-}\Pi_-)+\varrho_\alpha^{3,+}S]\}$.
\smallbreak\item
$a_{2,\alpha}^{bd}(y,F,D,\BB)=
\Tr\{F_2[\varrho_\alpha^{4,+}\Pi_++\varrho_{\alpha}^{4,-}\Pi_-]$
\smallbreak$
+F_1[L_{aa}(\varrho_\alpha^{5,+}\Pi_++\varrho_\alpha^{5,-}\Pi_-)+\rho_\alpha^{6,+}S]$\smallbreak$
+F_0[\rho_{mm}(\varrho_\alpha^{7,+}\Pi_++\varrho_\alpha^{7,-}\Pi_-)
+L_{aa}L_{bb}(\varrho_\alpha^{8,+}\Pi_++\varrho_\alpha^{8,-}\Pi_-)$\smallbreak\qquad$
+L_{ab}L_{ab}(\varrho_\alpha^{9,+}\Pi_++\varrho_\alpha^{9,-}\Pi_-)
+E(\varrho_\alpha^{10,+}\Pi_++\varrho_\alpha^{10,-}\Pi_-)$\smallbreak\qquad$
+\tau(\varrho_\alpha^{11,+}\Pi_++\varrho_\alpha^{11,-}\Pi_-)
+\varrho_\alpha^{12,+}L_{aa}S+\varrho_\alpha^{13,+}S^2+\varrho_\alpha^{14,+}\chi_{;a}\chi_{;a}]\}$.
\end{enumerate}
\end{lemma}

\subsection{Product formulas}
The following is a useful observation.

\begin{lemma}\label{lem-2.3}
 Let $M=M_1\times M_2$ where $M_1$ is a closed Riemannian manifold and $M_2$ is a Riemannian manifold with smooth boundary.. Let
$D_i$ be operators of Laplace type on
$M_i$ and let
$\BB$ be a mixed boundary operator on $M_2$ which we extend to $M$. Then:
$$
a_{\ell,\alpha,i}^{bd}((x_1,y_2),D,\BB)=\sum_{2k+j=\ell}a_k(x_1,D_1)a_{j,\alpha,i}^{bd}(y_2,D_2,\BB)\,.
$$
\end{lemma}

\begin{proof} Because the structures decouple, we have that $e^{-tD_\BB}=e^{-tD_1}e^{-tD_{2,\BB}}$. Let
$F(x_1,x_2)=F_1(x_1)F_2(x_2)$. Then:
$$
  \operatorname{Tr}_{L^2(M)}\{Fe^{-tD_\BB}\}=\operatorname{Tr}_{L^2(M_1)}\{F_1e^{-tD_\BB}\}
  \operatorname{Tr}_{L^2(M_2)}\{F_2e^{-tD_{2,\BB}}\}\,.
$$
The desired result then follows by equating terms in the asymptotic series.
\end{proof}

\subsection{Dimension shifting}
A-priori, the constants in Lemma \ref{lem-2.2} depend on the dimension. Fortunately, that is not the case.

\begin{lemma}\label{lem-2.4}
The constants $\rho_\alpha^{i,\pm}$  are independent of the dimension $m$.
\end{lemma}

\begin{proof} Let $M_1:=S^1$ and $D_1:=-\partial_\theta^2$ where $\theta$ is the usual
periodic parameter on the circle. Since
$M_1$ is a closed manifold, there are no boundary invariants. Since the structures are flat, $a_n(\theta,D_1)=0$ for $n>0$. By
Theorem
\ref{thm-1.2}, $a_0(\theta,D_1)=1$. Consequently, by Lemma \ref{lem-2.3},
$$a_{\ell,\alpha,i}^{bd}((\theta,y_2),D,\BB)=a_{\ell,\alpha,i}^{bd}(y_2,D_2,\BB)\,.$$
It now follows that $\rho_\alpha^{i,\pm}$ in dimension $m$ is
equal to $\rho_\alpha^{i,\pm}$ in dimension $m+1$.
\end{proof}

\subsection{The coefficients of $E$ and of $\tau$} In the proof of Lemma \ref{lem-2.4}, we applied Lemma \ref{lem-2.3}
with $M_1=S^1$. We take a product with $S^2$ to establish:
\begin{lemma}\label{lem-2.5}
We have the relations $\rho_\alpha^{10,\pm}=\rho_\alpha^{0,\pm}$ and $\rho_\alpha^{11,\pm}=\frac16\rho_\alpha^{0,\pm}$.
\end{lemma}
\begin{proof} Let $M_1=S^2$ be the sphere of radius
$\varepsilon$ in
$\mathbb{R}^3$, let $\Delta$ be the scalar Laplacian on $S^2$, and let $D_1:=\Delta-\delta$. We have
$\tau=\varepsilon^{-2}$ and
$E=\delta$. We apply Theorem \ref{thm-1.2} to see
$$a_0(x,D)=1\quad\text{and}\quad a_1(x,D)=\delta+\textstyle\frac16\varepsilon^{-2}\,.$$
Let $D_2$ be an operator of Laplace type on $M_2$ and let $\BB$ be mixed boundary conditions. We form
$M:=M_1\times M_2$ and $D=D_1+D_2$. We omit terms which involve neither $\delta$ nor $\varepsilon$ and use Lemma
\ref{lem-2.2} and Lemma
\ref{lem-2.3} to see
\begin{eqnarray*}
&&a_{2,\alpha,0}^{bd}((x_1,y_2),D,\BB)\\
&=&\Tr\{\delta(\varrho_\alpha^{10,+}\Pi_++\varrho_\alpha^{10,-}\Pi_-)
+\varepsilon^{-2}(\varrho_\alpha^{11,+}\Pi_++\varrho_\alpha^{11,-}\Pi_-)+...\}\\
&=&a_0(x_1,D_1)a_{2,\alpha,0}^{bd}(y_2,D_2,\BB)+a_1(x_1,D_1)a_{0,\alpha,0}^{bd}(y_2,D_2,\BB)\\
&=&\Tr\{(\delta+\textstyle\frac16\varepsilon^{-2})(\varrho_\alpha^{0,+}\Pi_++\varrho_\alpha^{0,-}\Pi_-)+...\}\,.
\end{eqnarray*}
Equating coefficients of $\delta$ and $\varepsilon^{-2}$ yields the desired identity.
\end{proof}

\subsection{Degree shifting} In Lemma \ref{lem-2.5}, we related $\rho_\alpha^{10,\pm}$ and $\rho_\alpha^{11,\pm}$
to $\rho_\alpha^{0,\pm}$. There are other relations of this form which are available:

\begin{lemma}
\ \begin{enumerate}
\item If $1\le i\le\ell$, then $a_{\ell,\alpha,i}=a_{\ell-1,\alpha-1,i-1}$.
\smallbreak\item $\varrho_\alpha^{1,\pm}=\varrho_{\alpha-1}^{0,\pm}$.
\smallbreak\item $\varrho_\alpha^{5,\pm}=\varrho_{\alpha-1}^{2,\pm}$ and $\varrho_\alpha^{6,+}=\varrho_{\alpha-1}^{3,+}$.
\smallbreak\item $\varrho_\alpha^{4,\pm}=\varrho_{\alpha-2}^{0,\pm}$.
\end{enumerate}\end{lemma}

\begin{proof} Choose $\chi(r)$ to be a smooth cut-off function which is  identically 0 near $r=\varepsilon$ and which
is identically 1 near
$r=0$. Let $F_\alpha^i(y,r):=\chi(r)r^{-\alpha}r^if(y)$; we then have $(F_\alpha^i)_j(y)=\delta_j^if(y)$. We suppress the
interior terms to express:
$$
\Tr_{L^2}(F_\alpha^ie^{-tD_\BB})\sim(4\pi)^{-m/2} t^{-(m-1)/2}\sum_{\ell=0}^\infty
t^{(\ell-\alpha)/2}\int_{\partial M}f(y)a_{\ell,\alpha,i}^{bd}(y)\dvol_{\partial M}(y)\,.
$$
Let $1\le i\le\ell$. We have $F_\alpha^i=F_{\alpha-1}^{i-1}$. Equating powers of $t$ in the asymptotic expansions then yields
the relation of Assertion (1). We apply Assertion (1) with $\ell=1$ and $i=1$ to derive Assertion (2); we apply Assertion (1)
with
$\ell=2$ and $i=1$ to derive Assertion (3); we apply Assertion (1) with $\ell=2$ and $i=2$ to see
$\rho_\alpha^{4,\pm}=\rho_{\alpha-1}^{1,\pm}$ and then apply Assertion (2) (after replacing $\alpha$ by $\alpha-1$) to
establish Assertion (4).
\end{proof}

\subsection{Relating pure Neumann and pure Dirichlet boundary conditions}
The following Lemma gives some relationships between the coefficients defining pure Neumann and pure Dirichlet boundary
conditions.

\begin{lemma}\label{lem-2.7}
$\varrho_\alpha^{0,+}+\varrho_\alpha^{0,-}=\varrho_\alpha^{7,+}+\varrho_\alpha^{7,-}=0$.
\end{lemma}

\begin{proof}
Let $M$ be the upper hemisphere of the unit sphere of $\mathbb{R}^{m+1}$. Let
$$T(x_1,...,x_{m+1})=(x_1,...,x_{m},-x_{m+1})$$
be an isometric involution of $S^m$ whose fixed point set is the boundary of $M$. Let
$$E(\lambda,\Delta_{S^m})=\{\phi\in C^\infty(S^m):\Delta_{S^m}\phi=\lambda\phi\}$$
be the eigenspaces of the spherical Laplacian on $S^m$. Let $(T^*\phi)(x):=\phi(Tx)$. Since $T$ is an isometry of $S^m$, it
commutes with the Laplacian and we may decompose
$$E(\lambda,\Delta_{S^m})=E^+(\lambda,\Delta_{S^m})\oplus E^-(\lambda,\Delta_{S^m})$$
into the $\pm1$ eigenvalues of $T^*$. It is then immediate that the elements of $E^-$ satisfy Dirichlet boundary conditions
while the elements of $E^+$ satisfy Neumann boundary conditions. If $p_\DD$ and $p_\NN$ are the
fundamental solutions of the heat equation of the Laplacians $\Delta_\DD$ and $\Delta_\NN$ on $M$ for Dirichlet and Neumann
boundary conditions, respectively, we may then conclude, after allowing for the renormalization of the $L^2$ norms of the
eigenvectors, that:
$$
p_\DD(x,x;t)+p_\NN(x,x;t)=2p_{S^m}(x,x;t)\quad\text{for}\quad x\in M\,.
$$
Since $S^m$ is a homogeneous space, there are constants $\tilde a_n$ so that
$$p_{S^m}(x,x;t)=p_{S^m}(t)\sim(4\pi t)^{-m/2}\sum_nt^n\tilde a_n\,.$$
Let $F\in C^\infty(S^m)$ satisfy $T^*F=F$. Then
$$
\Tr_{L^2}\{Fe^{-t\Delta_\DD}\}+\Tr_{L^2}\{Fe^{-t\Delta_\NN}\}\sim
(4\pi t)^{-m/2}\int_{S^m}F(x)\dvol(x)\cdot\sum_{n=0}^\infty t^n\tilde a_n\,.
$$
We may suppose $\alpha\notin\mathbb{Z}$. Since there are no $t^{-(m-1)/2}t^{(\ell-\alpha)/2}$ terms in the asymptotic
expansion of the right hand side of the above display, the boundary terms must vanish. The two relations of the Lemma now
follow.
\end{proof}

The coefficients $\rho_{\alpha}^{i,-}$ may be evaluated using Theorem \ref{thm-1.5} After changing notation appropriately to
simplify the relevant formulas, we summarize the results of this section in the following result:

\begin{lemma}\label{lem-2.8}
There exist universal constants $\{\vartheta_\alpha^i\}$ so that:
\begin{enumerate}
\item $a_{0,\alpha}^{bd}(y,F,D,\BB)=\kappa_{\alpha}\Tr\{F_0[\Pi_+-\Pi_-]\}dy$.
\smallbreak\item $a_{1,\alpha}^{bd}(y,F,D,\BB)=\kappa_{\alpha-1}
\Tr\{F_1[\Pi_+-\Pi_-]$
\smallbreak
$+F_0L_{aa}[\vartheta_\alpha^1\Pi_++\frac{\alpha-4}{2(\alpha-3)}\Pi_-]+\vartheta_\alpha^2F_0S\}dy$.
\smallbreak\item
$a_{2,\alpha}^{bd}(y,F,D,\BB)=k_{\alpha-2}\Tr\{F_2[\Pi_+-\Pi_-]$
\smallbreak$
+F_1[L_{aa}(\vartheta_{\alpha-1}^1\Pi_++\frac{\alpha-5}{2(\alpha-4)}\Pi_-)
    +\vartheta_{\alpha-1}^2S]$
\smallbreak$
+F_0[-\frac16\rho_{mm}(\Pi_+-\Pi_-)
+L_{aa}L_{bb}(\vartheta_\alpha^4\Pi_+-\frac{(\alpha-7)}{8(\alpha-6)}\Pi_-)$
\smallbreak\qquad$
+L_{ab}L_{ab}[\vartheta_\alpha^5\Pi_++\frac{\alpha-5}{4(\alpha-6)}\Pi_-]
+\vartheta_\alpha^6L_{aa}S+\vartheta_\alpha^7S^2$
\smallbreak\qquad$+
\frac1{3(1-\alpha)}(\tau+6E)(\Pi_+-\Pi_-)+\vartheta_\alpha^8\chi_{;a}\chi_{;a}]\}dy$.
\end{enumerate}
\end{lemma}

\section{Special case computations on the interval}\label{sect-3}
We note that $s\Gamma(s)=\Gamma(s+1)$. Since $\kappa_\alpha:=\frac12\Gamma(\frac{1-\alpha}2)$, we have the following
identities which we note for future reference:
\begin{equation}\label{eqn-3.a}
\frac{\kappa_\alpha}{\kappa_{\alpha-2}}=\frac2{1-\alpha},\quad
\frac{\kappa_{\alpha+1}}{\kappa_{\alpha-1}}=-\frac2\alpha\,.
\end{equation}

\begin{lemma}\label{lem-3.1}
We have that $\vartheta_\alpha^2=\frac4{1-\alpha}$ and that $\vartheta_\alpha^7=\frac{8}{(2-\alpha)(1-\alpha)}$.
\end{lemma}

\begin{proof}  Let $M=[0,1]$. Let $F_\alpha=r^{-\alpha}$
near
$x=0$ and $F_\alpha=0$ near $x=1$. Let $0\ne b\in\mathbb{R}$. We
form
$$
A:=\partial_x+b,\quad A^*:=-\partial_x+b,\quad D:=A^*A=AA^*=-(\partial_x^2-b^2)\,.
$$
Let $D_\DD$ and $D_\RR$ be the realizations, respectively, of $D$ with respect to Dirichlet boundary conditions and Robin
boundary conditions with $S(0)=b$, and $S(1)=-b$. Thus we may identify $\BBR\phi=A\phi|_{\partial M}$. We integrate by
parts to derive the Green's formula:
$$
\left.\int_0^1\bigg\{(A^*Au,v)-(u,A^*Av)\bigg\}(x)dx=\bigg\{-(Au,v)+(u,Av)\bigg\}\right|_0^1\,.
$$
This vanishes if $u=v=0$ on $\partial M$ or if $Au=Av=0$ on $\partial M$. Consequently both $D_\DD$ and $D_\RR$ are
self-adjoint. If
$D\phi=0$, then
$\phi^{\prime\prime}=b^2\phi$ so
$\phi=a_0e^{bx}+a_1e^{-bx}$. Thus $\phi$ satisfies Dirichlet boundary conditions means $\phi=0$ and thus $\ker(D_\DD)=\{0\}$.
Let
$$\left\{\theta_\nu,\lambda_\nu\right\}=\left\{\sin(\pi\nu x),2\pi^2\nu^2+b^2\right\}\quad\text{for}\quad \nu=1,2,...$$
be a spectral resolution for $D_\DD$. We have similarly that $\ker(D_\RR)=e^{-bx}\cdot\mathbb{R}$ and that
$\{A^*\theta_\nu/\sqrt{\lambda_\nu},\lambda_\nu\}$ is a spectral resolution of
$D_\RR$ on
$\ker(D_\RR)^\perp$. Let $p_\DD$ and $p_\RR$ be the fundamental solutions of the heat equation for $D_\DD$ and $D_\RR$,
respectively. We compute:
\begin{eqnarray*}
&&\partial_tp_\DD(x,x;t)=
-\sum_\nu\lambda_\nu e^{-t\lambda_\nu}\theta_\nu(x)^2=-\sum_\nu e^{-t\lambda_\nu}D\theta_\nu\cdot\theta_\nu,\\
&&\partial_tp_\RR(x,x;t)=-\sum_\nu e^{-t\lambda_\nu} A^*\theta_\nu\cdot A^*\theta_\nu\,.
\end{eqnarray*}
This then yields the identity:
\begin{eqnarray*}
&&2\partial_t\{p_\DD(x,x;t)-p_\RR(x,x;t)\}
=-2\sum_\nu e^{-t\lambda_\nu}\{D\theta_\nu\cdot\theta_\nu- A^*\theta_\nu\cdot
A^*\theta_\nu\}\\
&&\quad=2\sum_\nu
e^{-t\lambda_\nu}\{(\theta_\nu^{\prime\prime}\theta_\nu-b^2\theta_\nu\theta_\nu)
+(\theta_\nu^\prime\theta_\nu^\prime-2b\theta_\nu^\prime\theta_\nu+b^2\theta_\nu\theta_\nu)\}\\
&&\quad=\partial_x(\partial_x-2b)p_\DD(t;x,x)\,.\end{eqnarray*}
We suppose $\alpha<<0$ and $\alpha\ne\mathbb{Z}$ to ensure convergence and to ensure that the interior and the boundary
terms do not interact; the general case then follows by analytic continuation. We
integrate by parts to see:
\medbreak\qquad
$\displaystyle2\partial_t[\Tr_{L^2}\{F_\alpha e^{-tD_\DD}\}-\Tr_{L^2}\{F_\alpha e^{-tD_\RR}\}]$
\smallbreak\qquad\quad
$\displaystyle=\int_M2F_\alpha \partial_t(p_\DD(x,x;t)-p_\RR(x,x;t))\dvol(x)$
\smallbreak\qquad\quad
$\displaystyle=\int_MF_\alpha \partial_x(\partial_x-2b)p_\DD(x,x;t)\dvol(x)$
\smallbreak\qquad\quad
$\displaystyle=\int_M(F_\alpha ^{\prime\prime}+2bF_\alpha ^{\prime})p_\DD(x,x;t)\dvol(x)$
\smallbreak\qquad\quad
$\displaystyle=\Tr_{L^2}\left\{(F_\alpha ^{\prime\prime}+2bF_\alpha ^{\prime})e^{-t{D_\DD}}\right\}$.
\medbreak\noindent
Notice that $\partial_xF_\alpha=-\alpha F_{\alpha+1}$ and
$\partial_x^2F_{\alpha}=\alpha(\alpha+1)F_{\alpha+2}$ near the boundary of $M$. Since the underlying operator is the same,
the difference of the interior terms cancel and we have:
\medbreak\quad
$2\partial_t[\Tr_{L^2}\{F_\alpha e^{-tD_\DD}\}-\Tr_{L^2}\{F_\alpha e^{-tD_\RR}\}]$\medbreak\qquad
$\displaystyle\sim\sum_\ell(\ell-\alpha)t^{(\ell-\alpha-2)/2}\{a_{\ell,\alpha}^{bd}(F_\alpha,D,\BBD)
-a_{\ell,\alpha}^{bd}(F_\alpha,D,\BBR)\}$\medbreak\qquad
$=\Tr_{L^2}\{(F_{\alpha}''+2b F_{\alpha}')e^{-tD_\DD}\}$\medbreak\qquad
%$\sim\sum_nt^{n-\frac12}a_n(F_\alpha ^{\prime\prime}+2bF_\alpha ^{\prime},D)$\medbreak\qquad
$\displaystyle\sim \alpha(\alpha+1)\sum_jt^{(j-\alpha-2)/2}a_{j,\alpha+2}^{bd}(F_{\alpha+2},D,\BBD)$\medbreak\qquad
$\displaystyle-2\alpha b\sum_kt^{(k-\alpha-1)/2}a_{k,\alpha+1}^{bd}(F_{\alpha+1},D,\BBD)$.\medbreak\noindent
We equate coefficients in the asymptotic expansions to see
\begin{equation}\label{eqn-3.b}
\begin{array}{l}
\qquad(\ell-\alpha)\{a_{\ell,\alpha}^{bd}(F_\alpha,D,\BBD)-a_{\ell,\alpha}^{bd}(F_\alpha,D,\BBR)\}\\
=\alpha(\alpha+1)a_{\ell,\alpha+2}^{bd}(F_{\alpha+2},D,\BBD)
-2\alpha b a_{\ell-1,\alpha+1}^{bd}(F_{\alpha+1},D,\BBD)\,.\vphantom{\vrule height 11pt}
\end{array}\end{equation}
We remark that although the argument is
superficially similar to that used in Branson-Gilkey Lemma 3.2, the outcome is radically different owing to the necessity to
include the parameter $\alpha$; in particular, there is no interaction between the interior and the boundary terms.

We have $E=-b^2$ and $S=b$. We take $\ell=1$ in Equation (\ref{eqn-3.b}) to see:
\begin{eqnarray*}
&&(1-\alpha)\{a_{1,\alpha}^{bd}(F_\alpha,D,\BBD)-a_{1,\alpha}^{bd}(F_\alpha,D,\BBR)\}\\
&=&\alpha(\alpha+1)a_{1,\alpha+2}^{bd}(F_{\alpha+2},D,\BB_D)-2\alpha b a_{0,\alpha+1}(F_{\alpha+1},D,\BBD)\,.
\end{eqnarray*}
This leads to the identity:
$$(1-\alpha)\{-\kappa_{\alpha-1}\vartheta_\alpha^2\}b=2\alpha\{\kappa_{\alpha+1}\}b\,.$$
We apply Equation (\ref{eqn-3.a}) to see:
$$\vartheta_\alpha^2=\frac{-2\alpha}{1-\alpha}\frac{\kappa_{\alpha+1}}{\kappa_{\alpha-1}}
=\frac{-2\alpha}{1-\alpha}
\cdot\frac{-2}\alpha=\frac4{1-\alpha}\,.
$$
Finally, we take $\ell=2$ to see:
\begin{eqnarray*}
&&(2-\alpha)\{a_{2,\alpha}^{bd}(F_\alpha,D,\BBD)-a_{2,\alpha}^{bd}(F_\alpha,D,\BBR)\}\\
&=&\alpha(\alpha+1)a_{2,\alpha+2}^{bd}(F_{\alpha+2},D,\BBD)-2\alpha b a_{1,\alpha+1}^{bd}(F_{\alpha+1},D,\BB_D)\,.
\end{eqnarray*}
This gives rise to the identity:
\begin{eqnarray*}
&&(2-\alpha ) \kappa_{\alpha-2} b^2\left\{\frac4{(1-\alpha)}-\vartheta_\alpha^7\right\}=
\kappa_{\alpha} b^2 \left(\alpha (\alpha +1) \frac 6 {3 (1-[\alpha +2])} \right).\end{eqnarray*}
We use Equation (\ref{eqn-3.a}) to solve this identity for $\vartheta_\alpha^7$ to see:
\medbreak\qquad
$\displaystyle\vartheta_\alpha^7=\frac4{1-\alpha}-\frac1{2-\alpha}\frac2{1-\alpha}\alpha(\alpha+1)\frac2{-\alpha-1}
=\frac8{(1-\alpha)(2-\alpha)}$.\end{proof}

\section{Absolute and relative boundary conditions}\label{sect-4}

We establish the following result by generalizing the $1$-dimensional construction of Lemma \ref{lem-3.1} to the $2$-dimensional
setting.

\begin{lemma}\label{lem-4.1} We have that
\begin{enumerate}
\item
$\vartheta_\alpha^1=\frac{\alpha^2-\alpha-4}{2(\alpha-1)(3-\alpha)}$.
\smallbreak\item
$(1-\alpha)^2\{\vartheta_\alpha^6-8\vartheta_\alpha^8\}-2\alpha(\alpha+1)\{\vartheta_{\alpha+2}^4+\vartheta_{\alpha+2}^5\}
=\textstyle\frac{4(\alpha-1)}{\alpha-2}+\frac{\alpha(\alpha+1)(\alpha-1)}{4(\alpha-4)}$.
\end{enumerate}
\end{lemma}

\begin{proof}We modify an argument from McKean and Singer \cite{MS67}.
Let $M$ be a Riemann surface. Let $\Delta_a$ denote the Laplacian with absolute boundary conditions $\BBa$ as discussed above. Let
$\Lambda^{\operatorname{ev}}:=\Lambda^0\oplus\Lambda^2$ and let $\Lambda^{\operatorname{od}}:=\Lambda^1$. Let
$\{\theta_\nu,\lambda_\nu\}$ be a spectral resolution of
$\Delta_a^{\operatorname{ev}}$ on
$\ker(\Delta_a^{\operatorname{ev}})^\perp$. Then
$$\{(d+\delta)\theta_\nu/\sqrt{\lambda_\nu},\lambda_\nu\}$$ is a spectral resolution of
$\Delta_a^{\operatorname{od}}$ on $\ker(\Delta_a^{\operatorname{od}})^\perp$. One has:
\begin{eqnarray*}
&&2\partial_t\{p_{\Delta_a^{ev}}(x,x;t)-p_{\Delta_\BB^{od}}(x,x;t)\}\\
&=&2\partial_t\sum_{\nu:\lambda_\nu>0}
e^{-t\lambda_\nu}\left\{(\theta_\nu,\theta_\nu)-((d+\delta )\theta_\nu),(d+\delta)\theta_\nu)/\lambda_\nu\right\}\\ &=&-2\sum_{\nu:\lambda_\nu>0}
e^{-t\lambda_\nu}\{\lambda_\nu(\theta_\nu,\theta_\nu)-((d+\delta)\theta_\nu),(d+\delta)\theta_\nu)\}\\ &=&-2\sum_{\nu:\lambda_\nu>0}
e^{-t\lambda_\nu}\{(\Delta_a^{\operatorname{ev}}\theta_\nu,\theta_\nu)-((d+\delta)\theta_\nu,(d+\delta)\theta_\nu)\}\,.
\end{eqnarray*}
Now comes a crucial point. The Laplacian $\Delta_a$ decomposes as the direct sum of two scalar operators on
$\Lambda^{\operatorname{ev}}$. Let $\theta$ be a function. Then
\begin{eqnarray*}
&&2\Delta_0\theta=-2\theta_{;ii},\\
&&2(d+\delta)\theta=2 d\theta=2\theta_{;1}e^1+2\theta_{;2}e^2,\\
&&2\{(\Delta_0\theta,\theta)-((d+\delta)\theta,(d+\delta)\theta)\}=-2\theta_{;ii}\theta-2\theta_{;i}\theta_{;i}
=-(\theta^2)_{;ii}\\
&&\qquad=\Delta_0(\theta,\theta)\,.
\end{eqnarray*}
Next, let $\Theta=\theta e^1\wedge e^2$ be a $2$-form. We compute:
\begin{eqnarray*}
&&2\Delta_2\Theta=-2\theta_{;ii}e^1\wedge e^2,\\
&&2(d+\delta)\Theta=2\delta\Theta=-2\theta_{;1}e^2+2\theta_{;2}e^1,\\
&&2\{(\Delta_2\Theta,\Theta)-((d+\delta)\Theta,(d+\delta)\Theta)\}
=-2\theta_{ii}\theta-2\theta_{;i}\theta_{;i}
=\Delta_0(\theta^2)\\&&\qquad=\Delta_0(\Theta ,\Theta)\,.
\end{eqnarray*}
Consequently we have
$$
-2\sum_{\nu:\lambda_\nu>0} e^{-t\lambda_\nu}\{(\Delta_a^{\operatorname{ev}}\theta_\nu,\theta_\nu)-((d+\delta)\theta_\nu,(d+\delta)\theta_\nu)\}
=-\sum_{\nu:\lambda_\nu>0}
e^{-t\lambda_\nu}\Delta_0\{(\theta_\nu,\theta_\nu)\}\,.
$$
We suppose $\alpha<<0$ and $\alpha\ne\mathbb{Z}$. We may then integrate by parts to see:
\begin{eqnarray*}
&&\sum_\ell(\ell-\alpha-1)t^{(\ell-\alpha)/2-1}\left\{a_{\ell,\alpha}^{bd}(F_\alpha,\Delta^{\operatorname{ev}},\BBa)
    -a_{\ell,\alpha}^{bd}(F_\alpha,\Delta^{\operatorname{od}},\BBa)\right\}\\
&\sim&-\sum_ka_{k,\alpha+2}^{bd}(\Delta F_\alpha ,\Delta^{\operatorname{ev}},\BBa)t^{(k-\alpha-2)/2}\,.
\end{eqnarray*}
Equating terms in the asymptotic expansion then yields
\begin{eqnarray*}
&&(\ell-\alpha-1)\left\{a_{\ell,\alpha}^{bd}(F_\alpha,\Delta^{\operatorname{ev}},\BBa)
    -a_{\ell,\alpha}^{bd}(F_\alpha,\Delta^{\operatorname{od}},\BBa)\right\}\\
&=&-a_{\ell,\alpha+2}^{bd}(\Delta F_\alpha,\Delta^{ev},\BB_a)\,.
\end{eqnarray*}
We specialize to the case $M$ is the disk of radius $1$ in $\mathbb{R}^2$. Introduce the usual coordinates $(R,\theta)$ so that
$x=R\cos\theta$ and $y=R\sin\theta$; the distance to the boundary is then given by $r=1-R$. We have
\medbreak\qquad
$F_\alpha=(1-R)^{-\alpha}$,
\medbreak\qquad
$\Delta=-\partial_R^2-R^{-1}\partial_R-R^{-2}\partial_{\theta}^2$,
\medbreak\qquad
$\Delta F_\alpha=-\alpha(\alpha+1)F_{\alpha+2}-\alpha R^{-1}F_{\alpha+1}$.
\medbreak\noindent

We have that $R^{-1}=(1-r)^{-1}=1+r+...$. Since only the first 3 terms in the Taylor series expansion of $\Delta F_\alpha$ play a role
for $\ell=1,2$, we obtain the identity:
\begin{eqnarray*}
&&(\ell-\alpha-1)\{a_{\ell,\alpha}^{bd}(F_\alpha,\Delta^{\operatorname{ev}},\BBa)
    -a_{\ell,\alpha}^{bd}(F_\alpha,\Delta^{\operatorname{od}},\BBa)\}\\
&=&\alpha a_{\ell,\alpha+2}^{bd}((\alpha+1)F_{\alpha+2}+F_{\alpha+1}+F_\alpha,\Delta^{\operatorname{ev}},\BBa)\quad\text{for}\quad
\ell=1,2\,.
\end{eqnarray*}

We first set $\ell=1$. After canceling the factors of
$-\alpha$ from both sides of the equation we get the relation
$$
a_{1,\alpha}^{bd}(F_\alpha,\Delta^{\operatorname{ev}},\BBa)
    -a_{1,\alpha}^{bd}(F_\alpha,\Delta^{\operatorname{od}},\BBa)=-a_{1,\alpha+2}((\alpha+1)F_{\alpha+2}+F_{\alpha+1},\Delta^{ev},\BB_a)\,.
$$
The operator $\BB_a$ on functions is the Neumann boundary operator and the operator $\BB_a$ on $2$-forms is the Dirichlet boundary
operator. Near the boundary, we decompose a smooth $1$-form $\Theta=\Theta_1dr+\Theta_2 d\theta$. The operator $\BB_a$ on
$\Theta_1dr$ is the Dirichlet boundary operator and the operator $\BB_a$ on $\Theta_2d\theta$ is the Robin boundary operator with $S=-L$.
Substituting this into Lemma
\ref{lem-2.8} yields the relation:
$$\kappa_{\alpha-1}\vartheta_\alpha^2=
-\kappa_{\alpha+1}(\alpha+1)\left\{\frac12\frac{\alpha-2}{\alpha-1}+\vartheta_{\alpha+2}^1\right\}\,.$$
We continue our computation:
\begin{eqnarray*}
&&2(\alpha+1)\vartheta_{\alpha+2}^2=\frac{4\alpha+(\alpha+1)(\alpha-2)}{1-\alpha}=\frac{\alpha^2+3\alpha-2}{1-\alpha},\\
&&\vartheta_{\alpha+2}^1=\frac{\alpha^2+3\alpha-2}{2(\alpha+1)(1-\alpha)},\\
&&\vartheta_\alpha^1=\frac{(\alpha-2)^2+3(\alpha-2)-2}{2(\alpha-1)(3-\alpha)}
=\frac{\alpha^2-\alpha-4}{2(\alpha-1)(3-\alpha)}\,.
\end{eqnarray*}

Next we take $\ell=2$. This yields the relation:
\begin{eqnarray*}
&&(1-\alpha)\{a_{2,\alpha}^{bd}(F_\alpha,\Delta^{ev},\BB_a)-a_{2,\alpha}^{bd}(F_\alpha,\Delta^{od},\BB_a)\}\\
&=&\alpha a_{2,\alpha+2}^{bd}((\alpha+1)F_{\alpha+2}+F_{\alpha+1}+F_\alpha,\Delta^{ev},\BB_a)\,.
\end{eqnarray*}
Applying Lemma \ref{lem-2.8} with $S=-1$, $L_{aa}=L_{ab}L_{ab}=L_{aa}L_{bb}=1$ and $\chi_{:a}\chi_{:a}=8$ yields the relation:
$$
\begin{array}{l}(1-\alpha)\kappa_{\alpha-2}\{\vartheta_\alpha^6-\vartheta_\alpha^7-8\vartheta^8_\alpha\}\\
=\alpha\kappa_{\alpha}\left\{\vartheta_{\alpha+1}^1+\frac{\alpha-3}{2(\alpha-2)}+(\alpha+1)(\vartheta_{\alpha+2}^4-\frac{\alpha-5}{8(\alpha-4)}
+\vartheta_{\alpha+2}^5+\frac{\alpha-3}{4(\alpha-4)}\right\}\,.
\end{array}$$
Since $\frac{k_\alpha}{k_{\alpha-2}}=\frac2{1-\alpha}$,
$\vartheta_{\alpha+1}^1=\frac{\alpha^2+\alpha-4}{2\alpha(2-\alpha)}$, and $\vartheta_\alpha^7=\frac8{(2-\alpha)(1-\alpha)}$, we
have
$$
\begin{array}{l}
(1-\alpha)^2\{\vartheta_\alpha^6-\frac8{(2-\alpha)(1-\alpha)}-8\vartheta^8_\alpha\}\\
2\alpha\left\{\frac{\alpha^2+\alpha-4}{2\alpha(2-\alpha)}+\frac{\alpha-3}{2(\alpha-2)}+(\alpha+1)\left(\vartheta_{\alpha+2}^4-\frac{\alpha-5}{8(\alpha-4)}
+\vartheta_{\alpha+2}^5+\frac{\alpha-3}{4(\alpha-4)}\right)\right\}\,.\vphantom{\vrule height 11pt}
\end{array}$$
This leads to the relation:
\begin{eqnarray*}
&&(1-\alpha)^2\{\vartheta_\alpha^6-8\vartheta_\alpha^8\}-2\alpha(\alpha+1)\{\vartheta_{\alpha+2}^4+\vartheta_{\alpha+2}^5\}\\
&=&\textstyle\frac{8(\alpha-1)-\alpha^2-\alpha+4+\alpha(\alpha-3)}{\alpha-2}
+\frac{\alpha(\alpha+1)(\alpha-1)}{4(\alpha-4)}\\
&=&\textstyle\frac{4(\alpha-1)}{\alpha-2}+\frac{\alpha(\alpha+1)(\alpha-1)}{4(\alpha-4)}.
\end{eqnarray*}
The desired result now follows.
\end{proof}

\section{The pseudo-differential calculus}\label{sect-5}
In this section, we will use the pseudo-differential calculus to complete the calculation. Only the
invariant $\chi_{:a}\chi_{:a}$ genuinely involves a vector valued context; it will be determined by
Lemma \ref{lem-4.1} once the remaining coefficients are determined. Thus we will restrict our
attention to the case in which
$$D=\D=- g^{\mu\nu} \partial _\mu
\partial _\nu + b^\mu \partial _\mu\,,$$ is the scalar Laplacian. We shall work with Robin boundary
conditions.

We begin by reviewing some fairly standard material. Let
$\al=(\alpha_1,...,\alpha_m)$ be a multi-index. We set:
$$\begin{array}{ll}
 |\al | = \alpha _1 + ...+\alpha _m,&
\al! = \alpha _1 ! \times ...\times \alpha_m !, \\
x^{\al} = x_1^{\alpha_1} \times ... \times x_m ^{\alpha_m}, &
d_x^\al = \left( \frac \partial {\partial x_1} \right)^{\alpha _1}
\times ... \times \left( \frac \partial {\partial x_m}
\right)^{\alpha _m}, \\ D_{\al} ^x = (-\sqrt{-1})^{|{\al}|} d_x^{\al}\,.
\end{array}$$
We apologize in advance for the slight notational confusion involved with using $\alpha$ to control the growth of $F$ and also
to using ${\al}$ as a multi-index. We use the metric to raise and lower indices;  ``$,$" will denote
partial differentiation.

We refer to \cite{ DW92, DOP82, G94,Gr68,Se69} for additional
material about pseudo-differential operators.
We wish to construct the resolvent
$(\D-\lambda)^{-1}$ for large $\lambda$. We first suppose $M$ is a closed manifold. In the evaluation
of the heat equation asymptotics homogeneity properties of symbols
are relevant and it turns out that collecting terms according to homogeneity is useful; the
complex parameter $\lambda$ has weight $2$. Expand the symbol of $\Delta_M-\lambda$
in the form $a_2(x,\xi,\lambda)+a_1(x,\xi)+a_0(x,\xi)$ where:
\medbreak\qquad\qquad
$a_2 (x,\xi , \lambda ) =
g^{\mu\nu} \xi_\mu \xi _\nu - \lambda \equiv |\xi|^2 -\lambda$,
\smallbreak\qquad\qquad
$a_1 (x,\xi ,\lambda ) = \sqrt{-1} b^\mu \xi_\mu$,\quad\text{and}\quad
$a_0 (x,\xi ,\lambda )=0$.

\medbreak We formally expand the symbol of the resolvent in an asymptotic series:
\begin{eqnarray}\sigma ((\D-\lambda )^{-1}) (x,\xi , \lambda )
\sim \sum_{l=0} ^\infty q_{-2-l} (x,\xi ,\lambda )\,.\label{eqn-5.a}
\end{eqnarray}
The $q_k$ are then determined by the recursive relations:
\begin{equation}\label{eqn-5.b}
\begin{array}{l}
1 = a_2 (x,\xi , \lambda ) q_{-2} (x,\xi
,\lambda) ,\phantom{a_{\vrule height 12pt}}\\
0 = \displaystyle\sum_{k=2+l+|\al |-j}\frac 1
{\al !} [d_\xi^\al a_j (x,\xi , \lambda ) ] \,\, [ D_\al ^x
q_{-2-l} (x,\xi ,\lambda ) ] \quad \quad \mbox{for }k\geq
1.\end{array}
\end{equation}

To complete the proof of Theorem \ref{thm-1.6}, we must examine
$q_{-2}$, $q_{-3}$, and $q_{-4}$. We summarize the facts we shall need and omit details in the
interests of brevity -- the fact that $\D$ is scalar plays an essential role. Let
greek indices range from $1$ through $m$. We have:
\medbreak\qquad
$q_{-2}=a_2^{-1}$,
\smallbreak\qquad
$q_{-3} = -a_1 q_{-2}^2 +
c_{-3,3} q_{-2}^3$,
\smallbreak\qquad
$q_{-4}=-a_0 q_{-2}^2 + c_{-4,3} q_{-2}^3 + c_{-4,4} q_{-2}^4 +
c_{-4,5} q_{-2}^5$,
\medbreak\noindent
 where
\begin{eqnarray}
&&c_{-3,3} = -\sqrt{-1}\,(\partial
_\xi ^\nu a_2) (\partial
_\nu^x a_2) ,\nn\\
&&c_{-4,3} = a_1^2 - \sqrt{-1}\,(\partial _\xi ^\nu a_1) (\partial _\nu ^x
a_2 ) - \sqrt{-1}\,(\partial_\xi ^\nu a_2) (\partial _\nu ^x a_1 ) - {\textstyle\frac12} (\partial _\xi ^{\nu\mu} a_2 )
(\partial _{\nu\mu} ^x a_2 ) ,
\nn\\
&&c_{-4,4} = -3 a_1 c_{-3,3} + \sqrt{-1}\,(\partial _\xi^\nu a_2 )
(\partial _\nu ^x c_{-3,3} ) + (\partial _\xi ^{\nu\mu} a_2 )
(\partial _\nu ^x a_2 ) (\partial _\mu^x a_2) , \nn\\
&&c_{-4,5} = 3 c_{-3,3}^2.\nn
\end{eqnarray}
One has that:
\begin{eqnarray*}
&&q_{-2} (x ,\xi ,\lambda )= \textstyle\frac 1 {|\xi|^2 -\lambda } ,
\\
&&q_{-3} (x , \xi , \lambda) =\textstyle - \frac {1} { (|\xi|^2
-\lambda )^2}\sqrt{-1}b^\mu \xi_\mu  - \frac {1}{(|\xi|^2 -\lambda )^3} 2\sqrt{-1}g^{\sigma  \gamma }_{,\nu} \xi^\nu \xi_\sigma
\xi_\gamma,\\
&&q_{-4} (x, \xi , \lambda )=
\\
& &\quad\textstyle\frac 1 {(|\xi|^2 -\lambda )^3} \left\{ - b^\mu b^\nu \xi_\mu
\xi _\nu + b^{\nu } g^{\sigma   \beta}_{,\nu} \xi _\sigma  \xi _\beta + 2
b^\sigma  _{,\nu} g^{\nu\beta} \xi_\beta \xi_\sigma  - g^{\sigma  \beta}
_{,\nu\mu} g^{\nu\mu} \xi _\sigma  \xi _\beta \right\} \\
& &\quad\textstyle+\frac 1 {(|\xi|^2 -\lambda )^4} \left\{ -6 b^\mu g^{\sigma  \gamma}
_{,\nu} g^{\nu\beta} \xi _\mu \xi _\beta  \xi_\sigma  \xi _\gamma + 4
g^{\sigma  \gamma }_{,\beta \nu} g^{\beta \mu} g^{\nu\delta} \xi _\mu
\xi _\sigma  \xi _\gamma \xi _\delta \right.\\
& &\qquad\left.+4 g^{\sigma  \gamma }_{,\beta } g^{\beta
\mu}_{,\nu} g^{\nu\delta} \xi _\mu \xi _\sigma  \xi _\gamma \xi
_\delta + 2 g^{\sigma  \beta} _{,\nu} g^{\gamma \delta } _{,\mu}
g^{\nu\mu} \xi
_\sigma  \xi _\beta \xi _\delta \xi _\gamma\right\}\nn\\
& &\quad\textstyle+\frac 1 {(|\xi|^2 -\lambda )^5} \left\{ -12 g^{\sigma  \gamma}
_{,\nu} g^{\nu\beta} g^{\delta\tau}_{,\mu} g^{\mu\rho} \xi_\beta
\xi_\sigma  \xi_\gamma\xi_\rho \xi_\delta \xi_\tau\right\}.\nn
\end{eqnarray*}

If the manifold has a boundary the expansion (\ref{eqn-5.a}) has to be
augmented by a boundary correction. To formulate the conditions to
be satisfied by the boundary correction we expand about $r=0$. One may express the metric on the collar $\mathcal{C}_\varepsilon$ in the form
$$ds^2_M=g_{\sigma\varrho}(y,r)dy^\sigma\circ dy^\varrho+dr^2.$$
 The coordinate $y$ locally parametrizes the boundary, and
$r$ is the geodesic distance to the boundary, so $x=(y,r)$. A
tilde above any quantity will indicate that it is to be evaluated
at the boundary, that is at $r=0$. Furthermore, we use $\xi
=(\omega ,\tau )$.

We find $$\D-\lambda = \sum_{k=0}^\infty \frac 1 {k!} r^k
\sum_{|\al | \leq 2} \frac{\partial ^k} {\partial r^k} a_\al
(y,r)\bigg |_{r=0} D_{y,r}^\al $$ with the notation $$D_{y,r}^\al
= \left(\prod_{i=1}^{m-1} D_{y_i}^{\alpha_i}\right) D_r^{\alpha
_m}.$$ Introducing
\begin{eqnarray} a_j ( \vard ) =\left\{\begin{array}{ll}
\sum_{|\al | =j} a_\al (y,r) \left( \prod_{i=1}^{m-1}
\omega_i^{\alpha _i} \right) D_r^{\alpha_m}&\text{for}\quad j=0,1,\\
\sum_{|\al | =2} a_\al (y,r) \left( \prod_{i=1}^{m-1}
\omega_i^{\alpha _i} \right) D_r^{\alpha_m}-\lambda&\text{for}\quad j=2\end{array}\right. \nn
\end{eqnarray}
 we define the
partial symbol
\begin{eqnarray} \sigma ' (\D-\lambda ) = \sum_{k=0}^\infty
\frac 1 {k!} r^k \sum_{j=0}^2 \frac{\partial ^k } {\partial r^k}
a_j (\vard ) \bigg|_{r=0} .\nn
\end{eqnarray}
 As it turns out, the symbols
\begin{eqnarray}
a^{(j)} ( \vard ) = \sum_{l=0}^2
\sum_{\stackrel{k=0}{l-k=j}}^\infty \frac 1 {k!} r^k
\frac{\partial ^k}{\partial r^k} a_l ( \vard )\bigg |_{r=0} \nn
\end{eqnarray}
have suitable homogeneity properties and using these symbols we
write
\begin{eqnarray} \sigma ' (\D-\lambda ) = \sum_{j=-\infty} ^2 a^{(j)}
(\vard ) . \nn
\end{eqnarray}
 We write the symbol of the resolvent as
\begin{equation}\label{eqn-5.c}
\begin{array}{l}
\displaystyle\sigma
((\D-\lambda )^{-1}) (\var ) \\
\displaystyle=\sum_{j=0}^\infty q_{-2-j} (y,r,\omega ,\tau
,\lambda ) - e^{-\sqrt{-1}\tau r}\sum_{j=0}^\infty h_{-2-j} (\var ) ,
\end{array}\end{equation}
 where the second term is the boundary
correction. The factor $e^{-\sqrt{-1}\tau r}$ appears because the operator constructed from
these terms is the $Op^\prime(h)$ in \cite{Se69}, and $Op^\prime(h)=Op(he^{-\sqrt{-1}\tau r})$. This shows
\begin{eqnarray} \sigma ' (\D-\lambda )\circ
\sum_{j=0}^\infty h_{-2-j} (\var ) =0 .\nn
\end{eqnarray}
Here $\circ$ denotes the symbol product on $\mathbb{R}^{m-1}$. Analogously to
Equation (\ref{eqn-5.b}) this equation leads to the
differential equations
\begin{eqnarray} 0 &=& a^{(2)} ( \vard ) h_{-2} (\var ),
\nn\\ 0 &=& a^{(2)} (\vard ) h_{-2-j} (\var ) \nn\\
& &+ \sum_{\stackrel{\al , k , l<j}{j=l+2+|\al | -k}} \frac
1 {\al!} \left[ D_\omega ^\al a^{(k)} ( \vard ) \right]
\left[ (\iD^y)_\al h_{-2-l} (\var ) \right] .\nn
\end{eqnarray}
 For the
present considerations we need $h_{-2-j}$ for $j=0,1,2$, and we
have more
explicitly (repeated letters $a,b,c,...$ run over tangential coordinates $\{1,2,...,m-1\}$)
\begin{eqnarray*} 0 &=& a^{(2)} (\vard ) h_{-2} (\var ), \nn\\
0 &=&a^{(2)} (\vard ) h_{-3} (\var ) + a^{(1)} (\vard ) h_{-2}
(\var ) \nn\\
& & + \left[ D_\omega ^b a^{(2)} (\vard ) \right] \left[ (\iD^y)_b
h_{-2} (\var ) \right], \nn\\
0&=&a^{(2)} (\vard ) h_{-4} (\var ) + a^{(0)} (\vard ) h_{-2}
(\var ) \nn\\
& &+ \left[ D_\omega ^b a^{(1)} (\vard ) \right] \left[ (\iD^y)_b
h_{-2} (\var ) \right] \nn\\
& &+ {\textstyle\frac12} \left[ D_\omega ^{bc} a^{(2)}
(\vard ) \right] \left[ (\iD^y )_{bc} h_{-2} (\var ) \right] \nn\\
& &+ a^{(1)} (\vard ) h_{-3} (\var ) \\&&+ \left[ D_\omega ^b a^{(2)}
(\vard ) \right] \left[ (\iD^y ) _b h_{-3} (\var ) \right] .\nn
\end{eqnarray*}
The relevant equations for $a^{(i)} (\vard )$, $i=0,1,2$ are
\medbreak\quad
$a^{(2)} (\vard )$
\smallbreak\qquad
$= a_2 (\vard ) |_{r=0} =\gt \omega_a \omega_b + D_r^2 -\lambda $,
\medbreak\quad
$ a^{(1)} (\vard )$
\smallbreak\qquad
$= r (\partial_r a_2 (\vard ))|_{r=0} + a_1 (\vard ) |_{r=0}$
\smallbreak\qquad
$=r \gtr \omega_a \omega_b + \sqrt{-1} \tilde b^a \omega _a + \sqrt{-1} \tilde
b^r D_r $
\medbreak\quad
$a^{(0)} (\vard )$
\smallbreak\qquad
$= {\textstyle\frac12} r^2 (\partial_r^2 a_2 (\vard )
)|_{r=0} + r (\partial _r a_1 (\vard ) )|_{r=0}$
\smallbreak\qquad\qquad
$+ a_0 (\vard )
|_{r=0} $
\smallbreak\qquad
$= {\textstyle\frac12} r^2 \gtrr \omega_a \omega_b + r\sqrt{-1}\,\tilde b^a_{,r}
\omega_a + r \sqrt{-1}\, \tilde b^r _{,r} D_r$.
\medbreak\noindent

The differential equations have to be augmented by a growth
condition
\begin{equation}\label{eqn-5.d} h_{-2-j} (\var )
\to 0 \quad \quad \mbox{as}\quad r\to\infty ,
\end{equation}
and an initial condition corresponding to the Robin boundary
condition
$${\mathcal B} \phi = (\partial_r +S) \phi$$ considered here. The first few boundary symbols satisfy
\begin{eqnarray}\label{eqn-5.e}
\partial_r h_{-2} (\var ) |_{r=0} &=& \sqrt{-1} \tau q_{-2} (\var ) |_{r=0}\,,\nn\\
\partial_r h_{-3} (\var ) |_{r=0} &=& -S h_{-2} (\var ) |_{r=0} + S q_{-2} (\var ) |_{r=0} \nn\\
& &+ \sqrt{-1} \tau q_{-3} (\var ) | _{r=0} + \partial _r q_{-2} (\var ) |_{r=0} ,\nn\\
\partial_r h_{-4} (\var ) |_{r=0} &=& -S h_{-3} (\var ) |_{r=0} + S q_{-3} (\var ) |_{r=0} \nn\\
& &+ \sqrt{-1} \tau q_{-4} (\var ) | _{r=0} + \partial _r q_{-3} (\var ) |_{r=0} .
\end{eqnarray}
Once the
symbols $h_{-2-j}$ have been determined, their contribution to the
asymptotics of the trace of the heat kernel follows from multiple
integration. As before, we suppose $r^\alpha F\in C^\infty(\mathcal{C}_{\varepsilon})$.
The contribution reads $$\sum_{l=0}^\infty t^{\frac {1-\alpha -m}
2 } t^{\frac l 2} \int_{\partial M} \eta _{\frac l 2}
(y,F,\D)  dy$$ with
\begin{eqnarray}&& \eta _{\frac l 2} (y,F,\D )= \frac 1
{(2\pi)^{m+1}} \sum_{j+k=l}\nn
\int_{\mathbb{R}^{m-1}} d\omega
\int_{-\infty }^\infty ds\nn\\&&\quad\times \int_0^\infty d\bar r
e^{\sqrt{-1}\,s} \left( - \int_\gamma d\tau e^{-\sqrt{-1}\,\tau \bar r}\right)
h_{-2-j} (y,\bar r,\omega,\tau, -\sqrt{-1}\,s) \bar r ^{k-\alpha}F_k(y),\label {eqn-5.f}
\end{eqnarray}
 where $\gamma$ is anticlockwise enclosing
the poles of $h_{-2-j}$ in the lower half-plane. The integral with
respect to $s$ is the contour integral transforming the resolvent to
the heat kernel. Note that from (\ref{eqn-5.c}) the
contribution to the heat kernel is {\bf minus} the above.

As will become clear in the following, with $\Lambda =
\sqrt{|\omega|^2 +\sqrt{-1}\,s}$, we need integrals of the type
\medbreak\noindent\centerline{$\displaystyle
T_{ab...}^{kljn} \equiv \int_{\mathbb{R}^{m-1}} d\omega
\int_{-\infty }^\infty ds \int_0^\infty d\bar r
e^{\sqrt{-1}\,s} \left( - \int_\gamma d\tau e^{-\sqrt{-1}\,\tau \bar r}\right)
\frac{\tau ^k \bar r ^{l-\alpha} \omega_a \omega_b
...}{\Lambda^j (\tau^2 + \Lambda^2 ) ^n} e^{-\bar r\Lambda }
.\nn$
}\medbreak\noindent
 The $\tau$ integration can be done using
\begin{eqnarray}\int_\gamma d\tau e^{-\sqrt{-1}\,\tau \bar r} \frac{\tau
^k}{(\tau ^2 + \Lambda^2)^l} &=& \frac{(\sqrt{-1})^k (-1)^{l+k} \pi
}{(l-1)!} \left( \frac 1 {2\Lambda} \frac d
{d\Lambda}\right)^{l-1} \left[ \Lambda^{k-1} e^{-\bar r
\Lambda}\right] .\nn
\end{eqnarray}
 So
\begin{eqnarray*} T_{ab...}^{kljn} &=& {\textstyle\frac{(\sqrt{-1})^k
(-1) ^{n+k+1} \pi}{(n-1)!}}  \int_{\mathbb{R}^{m-1}} d\omega
\int_{-\infty}^\infty ds\\&&\quad\times \int_0^\infty d\bar r
e^{\sqrt{-1}\,s} \bar r ^{l-\alpha} {\textstyle\frac{\omega_a \omega_b ...} {\Lambda^j}}
e^{-\bar r \Lambda} \left( {\textstyle\frac 1 {2\Lambda} \frac d {d\Lambda}}
\right)^{n-1} \left[ \Lambda^{k-1} e^{-\bar r
\Lambda}\right].\nn
\end{eqnarray*}
Performing the $\Lambda$-differentiation,
different $\bar r$-dependent functions would occur. It is
therefore desirable to first perform the $\bar r$-integration
before performing the $\Lambda$-derivatives explicitly. This is
achieved by noting that ($z=\Lambda$ has to be put after
the $\Lambda$ differentiation has been performed)
\begin{eqnarray*} T_{ab...}^{kljn} &=& \left.\frac{(\sqrt{-1})^k (-1)
^{n+k+1} \pi}{(n-1)!} \int_{\mathbb{R}^{m-1}} d\omega
\int_{-\infty}^\infty ds e^{\sqrt{-1}\,s} \frac{\omega_a \omega_b
...} {\Lambda^j}\right.\\&&\quad\left.\times\left( \frac 1 {2\Lambda} \frac d {d\Lambda}
\right)^{n-1} \Lambda^{k-1} \int_0^\infty d\bar r \bar
r^{l-\alpha} e^{-\bar r (\Lambda + z)} \right|_{z=\Lambda}\nn\\
&=&\left.\frac{(\sqrt{-1})^k (-1) ^{n+k+1} \pi}{(n-1)!}\Gamma(l+1-\alpha )
\int_{\mathbb{R}^{m-1}} d\omega \int_{-\infty}^\infty ds
e^{\sqrt{-1}\,s} \frac{\omega_a \omega_b ...} {\Lambda^j}\right.\\&&\qquad\left. \left( \frac 1
{2\Lambda} \frac d {d\Lambda} \right)^{n-1}
\frac{\Lambda^{k-1}}{(\Lambda +
z)^{l+1-\alpha}}\right|_{z=\Lambda}.\nn
\end{eqnarray*}
We can proceed in general by introducing numerical multipliers $c_{nkl}$
according to
\begin{eqnarray} \left. \left( \frac 1 {2\Lambda} \frac d
{d\Lambda}\right)^{n-1} \frac{\Lambda^{k-1}}{(\Lambda
+z)^{l+1-\alpha }} \right|_{z=\Lambda} = c_{nkl} \frac 1
{\Lambda^{l+2n-k-\alpha}}.\nn
\end{eqnarray}
 The $s$-integration is then
performed using $$\int_{-\infty}^\infty ds
\frac{e^{\sqrt{-1}\,s}}{(|\omega|^2 +\sqrt{-1}\,s )^\beta} = \frac{2\pi} {\Gamma (\beta
)} e^{-|\omega|^2} .$$ The final $\omega$-integrations follow from
\begin{eqnarray} C(y) & \equiv & \int_{\mathbb{R} ^{m-1}} d\omega e^{-\gt
\omega_a \omega_b + \sqrt{-1}\, y^a \omega_a} = \pi^{\frac{m-1} 2}
\sqrt{\tilde g} e^{-\frac {\tilde g _{ab} y^a y^b} 4} , \nn
\end{eqnarray}
 by
observing that
\begin{eqnarray} \int_{\mathbb{R} ^{m-1}} d\omega \,\,
\omega_{a_1} \omega _{a_2} ... \omega_{a_r} e^{-\gt \omega_a
\omega_b} &=& \left. \left( \frac 1 {\sqrt{-1}} \right)^r \frac \partial
{\partial y^{a_1}} \cdot\cdot\cdot \frac \partial {\partial
y^{a_r}} C(y) \right|_{y=0}.\nn
\end{eqnarray}
 In particular
\begin{eqnarray}
\int_{\mathbb{R} ^{m-1}} d\omega \,\,
e^{-|\omega|^2} &=& \pi^{\frac{m-1} 2} \sqrt {\tilde g}, \nn\\
\int_{\mathbb{R} ^{m-1}} d\omega \,\,\omega_a \omega_b
e^{-|\omega|^2} &=& {\textstyle\frac12} \pi^{\frac{m-1} 2} \sqrt{\tilde g}
\tilde g_{ab}\nn,\\
\int_{\mathbb{R} ^{m-1}} d\omega \,\,\omega_a
\omega_b\omega_c\omega_d e^{-|\omega|^2} &=& {\textstyle\frac14}
\pi^{\frac{m-1} 2} \sqrt{\tilde g} \left( \tilde g_{ab} \tilde
g_{cd} + \tilde g _{ac} \tilde g _{bd} + \tilde g_{ad} \tilde
g_{bc} \right).\nn
\end{eqnarray}
 Introducing the numerical multipliers
$d_{kljn}$ according to
\begin{eqnarray} d_{kljn} = \frac {2 (\sqrt{-1})^k (-1)^{n+k+1}
\pi^2 \Gamma (l+1-\alpha ) c_{nkl}}{(n-1)! \Gamma \left( \frac{
j+l-k-\alpha } 2 +n \right)} , \nn
\end{eqnarray}
 we obtain the
compact-looking answers
\begin{eqnarray} T_{ab...}^{kljn} = d_{kljn}
\int_{\mathbb{R} ^{m-1}} d\omega \,\, \omega_a \omega_b ...
e^{-|\omega|^2},\nn
\end{eqnarray}
 where the last $\omega$-integration is
performed with the above results.

Note that the numerical multipliers $d_{kljn}$ are easily
determined using an algebraic computer program. Therefore, all
appearing integrals can be very easily obtained.

Let us apply this formalism explicitly to the leading orders, and
we start with $h_{-2} (\var ).$ The relevant differential equation
reads $$(\partial_r^2 - \Lambda ^2) h_{-2} (\var ) =0,$$ which has
the general solution $$h_{-2} (\var ) = g_1 e^{-r \Lambda} + g_2 e^{r
\Lambda} .$$ The asymptotic condition (\ref{eqn-5.d}) on the symbol
as $r\to\infty$ imposes $g_2=0$. The initial condition $\partial_r h_{-2}
|_{r=0} = \sqrt{-1} \tau q_{-2} |_{r=0}$ gives $g_1 = -\sqrt{-1} (\Lambda (\tau^2 +
\Lambda^2))^{-1} \tau$. Putting the information
together we have obtained
$$h_{-2} (\var ) = - \frac{\sqrt{-1} \tau} {\Lambda  (\tau ^2 + \Lambda^2)} e^{-r \Lambda}
.$$ Performing the relevant integrals, with the notation $$\int dI
= \int_{\mathbb{R} ^{m-1}} d\omega \int_{-\infty}
^\infty ds \int_0^\infty d\bar r e^{\sqrt{-1}\,s}\left( -
\int_\gamma d\tau e^{-\sqrt{-1}\,\tau \bar r}\right) \bar r
^{-\alpha} ,$$ produces
\begin{eqnarray*}
&&\int dI h_{-2} ( y,\bar r,\omega,\tau ,-\sqrt{-1}\,s ) = - \sqrt{-1}  d_{1011} \pi^{\frac{m-1} 2 } \sqrt
{\tilde g}\\& =& -\frac{2^\alpha \pi^2 \Gamma (1-\alpha) } { \Gamma
\left( 1-\frac \alpha 2 \right)} \pi ^{\frac{m-1} 2 } \sqrt
{\tilde g}= -\pi \Gamma \left( \frac{1-\alpha} 2 \right) \pi^{m/2}
\sqrt {\tilde g}\,.
\end{eqnarray*} Taking into account the prefactor in
(\ref {eqn-5.f}) and the change of sign, this agrees with Assertion (1) of Theorem \ref{thm-1.6}.

In the next order we obtain
\begin{eqnarray} (\partial _r^2 - \Lambda^2)
h_{-3} (\var ) = (E+U_1) e^{-r \lambda} + (F+U_2) r e^{-r \Lambda}
, \nn
\end{eqnarray}
 where
\begin{eqnarray*} E&=& \frac{ \sqrt{-1}\tilde b^r \tau}{\tl } ,
\quad \quad F= -\frac{\sqrt{-1}\tau\gtr \omega_a \omega_b} {\Lambda (\tl )} , \nn\\
U_1(\omega ) &=& \frac{ \tilde b^a \omega_a \tau} {\Lambda (\tl ) } +\frac{\tilde g ^{ac} _{,b} \omega_a \omega_c \omega^b \tau} {\Lambda^3 (\tl )} + \frac {2\,
\tilde g^{ac} _{,b} \omega^b\omega_a \omega_c \tau } {\Lambda (\tl )^2},\\
U_2 (\omega ) &= &\frac{ \tilde g^{ac} _{,b} \omega^b \omega_a
\omega_c \tau }{\Lambda^2 (\tl ) }  .
\end{eqnarray*}
Note, for later
arguments, that $U_1(\omega )$ and $U_2 (\omega )$ are odd functions
in $\omega$. Furthermore, for the scalar Laplacian at hand $b^a =
g^{bc} \Gamma_{bc}{}^a$; thus they contain only {\it tangential}
derivatives of the metric.

Using for example the annihilator method, we write down the
general form of the solution to this differential equation as
$$h_{-3} (\var ) = c_1 e^{-r \Lambda} + c_2 r e^{-r \Lambda} + c_3
r^2 e^{-r \Lambda} + c_4 e^{r\Lambda}.$$ From the asymptotic
condition (\ref{eqn-5.d}) we conclude $c_4 =0$. From the initial
condition given in Equation (\ref{eqn-5.e}) we obtain
\begin{eqnarray} c_1 &=& - \frac{1} {4 \Lambda^3} (F + U_2) - \frac 1 {2 \Lambda^2} (E+U_1) \nn\\
 & & - \frac{\sqrt{-1} S \tau} {\Lambda^2 (\tl )} - \frac{ S}{\Lambda (\tl )} - \frac{\tilde b^a \omega_a \tau} {\Lambda (\tl )^2} \nn\\
 & &-\frac{\tilde b^r \tau^2} {\Lambda (\tl )^2} - \frac{ 2 \tilde g^{ab}_{,c} \omega^c \omega _a \omega_b \tau}{ \Lambda (\tl )^3}
-\frac{2 \gtr \omega_a \omega_b \tau^2} {\Lambda  (\tl )^3} + \frac{\gtr \omega_a \omega_b }{\Lambda (\tl )^2} .\nn
\end{eqnarray}
From the differential
equation we derive
\begin{eqnarray} & & c_2 = - {\textstyle\frac 1 {4 \Lambda^2}} (F+U_2) -
{\textstyle\frac 1 {2\Lambda}} ( E+U_1) , \nn\\
& &c_3 = - {\textstyle\frac 1 {4 \Lambda}} (F+U_2) .\nn
\end{eqnarray}
Collecting the
available information, we see
\begin{eqnarray} h_{-3} (\var ) = D e^{-r
\Lambda} + Br e^{-r \Lambda} + C r^2 e^{-r \Lambda} + O(\omega ) ,
\nn
\end{eqnarray}
with
\medbreak\quad\qquad
$D= \frac{\sqrt{-1} \tau \gtra}{4 \Lambda^4 (\tl )}
    - \frac{\sqrt{-1} \tilde b ^r \tau}{2 \Lambda^2 (\tl )} - \frac{\sqrt{-1} S \tau} {\Lambda^2 (\tl )}$
\smallbreak\qquad\qquad\qquad
$- \frac{S}{\Lambda (\tl )} - \frac{\tilde b^r \tau^2}{\Lambda (\tl )^2}
    - \frac{ 2 \tau ^2 \gtra}{\Lambda (\tl )^3} + \frac{ \gtr \omega_a \omega_b  } {\Lambda (\tl )^2}$,
\medbreak\quad\qquad
$B= \frac{\sqrt{-1} \tau \gtra}{4 \Lambda^3 (\tl )}
    - \frac{ \sqrt{-1} \tilde b^r \tau } {2 \Lambda (\tl )}$,
\medbreak\quad\qquad
$C=  \frac{ \sqrt{-1} \tau \gtra } { 4 \Lambda^2 (\tl )}$,
\medbreak\noindent
and where $O(\omega )$ is an odd function in $\omega$.
Furthermore, $O(\omega )$ contains only tangential derivatives of
the metric. We next perform the multiple integrals; note, odd
functions in $\omega$ do not contribute. We obtain
\begin{eqnarray*} &&\int dI
h_{-3} (y,\bar r ,\omega,\tau,-\sqrt{-1}\,s )\\ &=& \pi^{\frac{m-1} 2 } \sqrt
{\tilde g} \gtr \tilde g_{ab} \left\{ -\textstyle \frac {\sqrt{-1}} 4 d_{1021} - {\textstyle\frac12} d_{2012} -
\frac{\sqrt{-1}} 4 d_{1111} + {\textstyle\frac{\sqrt{-1}} 8 d_{1041}}\right.\nn\\
 & &\hspace{2.0cm}\left.-\textstyle{ d_{2013} } + {\textstyle \frac 1 2 d_{0012}} + {\textstyle \frac{\sqrt{-1}} 8 d_{1131}} + {\textstyle \frac{\sqrt{-1}} 8 d_{1221}}   \right\}\nn\\
& &+S \sqrt{\tilde g} \pi^{\frac{m-1} 2 } ( -\sqrt{-1} d_{1021} - d_{0011} ) \nn\\
&=&\textstyle \sqrt {\tilde g} \pi^{\frac{m+2} 2} \Gamma \left( 1 - \frac \alpha 2 \right) \left(
\frac{\alpha^2-\alpha -4} {4 (\alpha -1) (\alpha -3)} \gtr\tg +\frac 4 {\alpha-1} S \right).
\end{eqnarray*}
This confirms the value of $\vartheta_\alpha^1$ in Lemma \ref{lem-4.1} and of $\vartheta_\alpha^2$ in Lemma \ref{lem-3.1} after taking into account the prefactor in (\ref
{eqn-5.f}) and the fact that $\gtr \tilde g_{ab}= - \tilde g^{ab} \tgr = 2
g^{ab}L_{ab}$.

Up to this point the calculation can be considered a warm up for
the next order. Leaving aside the $S$-terms for the moment, we would like to
determine the universal coefficients of the geometric invariants
$L_{aa} L_{bb}$ and $L_{ab} L_{ab}$. In terms of the
metric these are determined by
$$L_{ab} = - {\textstyle\frac12} \tgr\,.$$
Using the Christoffel symbols
$$\Gamma_{jk}{} ^i = {\textstyle\frac12}
g^{il} \left( g_{lj,k} + g_{kl,j} - g_{jk,l} \right),$$
and taking into account that with our sign convention the scalar curvature is given by the contraction $g^{jk}R_{ijk}{}^i$,
we may expand the Riemann curvature tensor in the form:
$$
R_{ijk}{}^l=\Gamma_{jk}{}^l{}_{,i}-\Gamma_{ik}{}^l{}_{,j}+\Gamma_{in}{}^l\Gamma_{jk}{}^n-\Gamma_{jn}{}^l\Gamma_{ik}{}^n\,.
$$
The normal projection of the Riemann curvature tensor
reads
\begin{eqnarray} \tilde R_{amma} &=& - {\textstyle\frac12} \tilde g^{ac} _{,r}
\tilde g_{ac,r} - {\textstyle\frac12} \tilde g^{ac}\tilde  g_{ac,rr} - {\textstyle\frac14} \tilde g^{bc}_{,r} \tilde g^{ad} _{,r} \tilde g_{ca}
\tilde g_{bd} \nn\\
&=& {\textstyle\frac14} \tilde g^{ab} \tilde g^{cd} \tilde g_{ac,r} \tilde
g_{bd,r} - {\textstyle\frac12} \tilde g^{ac}\tilde  g_{ac,rr}.\nn
\end{eqnarray}
The above results suggest a strategy for the calculation. It suffices to
consider the special case where the metric is independent of $y$. As
a consequence, our answer will have the form
$$(4 \pi )^{-m/2}
\left\{ H \tilde g^{ac} \tilde g_{ac,rr} + K\tilde  g^{ab} \tilde
g^{cd}\tilde  g_{ac,r}\tilde g_{bd,r} + L \tilde g^{ab} \tilde
g^{cd} \tilde g_{ab,r}\tilde g_{cd,r} \right\}
$$
plus terms involving $S$. This has to be compared with the terms in
$a_{2,\alpha}^{bd} (F,\Delta_M)$ that possibly contribute to these geometric
invariants. In detail one can show these terms are (mod terms with
tangential derivatives of the metric)
\medbreak\qquad$ \frac 1 {3 (1-\alpha)} \tilde \tau
  - \frac 1 6 \tilde \rho _{mm} + \vartheta_\alpha ^4 L_{aa} L_{bb} + \vartheta _\alpha ^5 L_{ab} L_{ab}$
\smallbreak\qquad\quad
$=\left( \frac 1 {12} - \frac 1 {3 (1-\alpha)} \right) \tilde g^{ac} \tilde g_{ac,rr}
  + \left( - \frac 1 {12(1-\alpha)} + \frac 1 4 \vartheta_\alpha ^4 \right) \tilde g^{ab} \tilde g^{cd} \tilde g_{ab,r} \tilde
g_{cd,r}$
\smallbreak\qquad\qquad
$ + \left( \frac 1 {4 (1-\alpha )} - \frac 1 {24}
+ \frac 1 4 \vartheta _\alpha ^5\right) \tilde g^{ab} \tilde g^{cd} \tilde g_{ac,r} \tilde g_{bd,r}$.
\medbreak\noindent So
once we know $H,K,L$, as a check we can verify that
$$H= \textstyle\left( \frac 1 {12} - \frac 1 {3 (1-\alpha )} \right) \kappa_{\alpha -2} ,$$ and we can deduce
$$
\textstyle
\vartheta_\alpha ^4 = \frac {4K} {\kappa_{\alpha -2}} + \frac 1 {3 (1-\alpha )}, \quad \quad \vartheta _\alpha ^5 = \frac{ 4L} {\kappa_{\alpha -2} } - \frac 1 {1-\alpha } + \frac 1 6 .
$$
In summary, when writing down
the differential equation for $h_{-4} (\var )$, we can neglect all
terms that are odd in $\omega$ as well as all terms that contain
tangential derivatives of the metric. We obtain (up to irrelevant
terms)
\begin{eqnarray} (\partial_r^2 -\Lambda ^2) h_{-4} (\var ) = M e^{-r
\Lambda} + N r e^{-r \Lambda } + P r^2 e^{-r \Lambda} + Q r^3
e^{-r \Lambda } ,\nn
\end{eqnarray}
where
\begin{eqnarray*}
M &=& -\Lambda \tilde b^r D + \tb B,\nn\\
N &=& \frac{\sqrt{-1} \tilde b^r_{,r} \tau} {\tl } + \gtra D - \Lambda \tb B + 2 \tb C,\nn\\
P &=& - \frac{\sqrt{-1} \gtrr \omega_a \omega_b \tau} {2\Lambda (\tl )} + \gtra B - \Lambda \tb C,\nn\\
Q &=& \gtra C,\nn
\end{eqnarray*}
with $B,C,D$ given above.

So the solution has the
form, taking into account the asymptotic behavior (\ref{eqn-5.d}),
\begin{eqnarray} h_{-4} (\var ) &=& \tilde \alpha e^{-r \Lambda} + \beta r
e^{-r \Lambda} + \gamma r^2 e^{-r \Lambda} + \delta r^3 e^{-r
\Lambda} + \epsilon r^4 e^{-r \Lambda} .\nn
\end{eqnarray}
To simplify the notation, let $\Xi:=\tau^2+\Lambda^2$. From the initial
condition we obtain, up
to irrelevant terms,
\begin{eqnarray} &&\tilde \alpha =\frac {\beta} \Lambda + \frac{SD} \Lambda + \frac{\sqrt{-1} \tb S
\tau}{\Lambda \Xi^2} + \frac{2\sqrt{-1} S \gtra \tau }{\Lambda \Xi^3} \nn\\ & &+
 \frac{\sqrt{-1} \tb \tb \tau^3 - \sqrt{-1} \tb \gtra \tau - 2 \sqrt{-1} \tilde b
^r _{,r} \tau^3 + \sqrt{-1} \gtrr \omega_a \omega_b \tau}{\Lambda \Xi^3} \nn\\ & &+
\frac{ 6 \sqrt{-1} \tb \gtra \tau^3 - 4 \sqrt{-1} \gtrr \omega_a \omega_b \tau^3 - 2 \sqrt{-1}
\gtr \tilde g ^{cd}_{,r} \omega_a \omega_b \omega_c \omega_d \tau }{\Lambda
\Xi^4} \nn\\ & &+ \frac {12
\sqrt{-1} \gtr \tilde g^{cd}_{,r} \omega_a \omega_b \omega_c \omega_d \tau^3}{\Lambda \Xi^5} +
\frac{\sqrt{-1} \tilde b^r_{,r} \tau}{\Lambda \Xi^2} - \frac{2\sqrt{-1} \tb \gtra \tau}{\Lambda
\Xi^3} \nn\\ & & + \frac{2 \sqrt{-1} \gtrr \omega_a \omega_b \tau}{\Lambda \Xi^3} - \frac{6
\sqrt{-1} \gtr \tilde g^{cd}_{,r} \omega_a \omega_b \omega_c \omega_d \tau}{\Lambda \Xi^4}.\nn
\end{eqnarray}
From the differential equation we obtain the conditions
$$\begin{array}{ll} M = - 2
\Lambda \beta + 2 \gamma,&
N = - 4 \Lambda \gamma + 6 \delta ,\\
P = - 6 \Lambda \delta + 12 \epsilon,&
Q = - 8 \epsilon \Lambda \,.\end{array}$$
This determines the numerical
multipliers $\beta$, $\gamma$, $\delta$ and $\epsilon$ to be
$$\begin{array}{ll}
\beta =\textstyle- {\textstyle\frac38} \frac Q {\Lambda^4} - {\textstyle\frac14} \frac P
{\Lambda^3} - {\textstyle\frac14} \frac N {\Lambda^2} - {\textstyle\frac12} \frac M
\Lambda,&
\gamma = \textstyle- {\textstyle\frac38} \frac Q {\Lambda^3} - {\textstyle\frac14} \frac P
{\Lambda^2} - {\textstyle\frac14} \frac N \Lambda,\\
\delta = \textstyle- {\textstyle\frac14} \frac Q {\Lambda^2} - {\textstyle\frac1{6}} \frac P
\Lambda,&
\epsilon =\textstyle - {\textstyle\frac18} \frac Q \Lambda.\phantom{\vrule height 12pt}
\end{array}$$
For $\Delta_M$, we have:
\medbreak\qquad
 $\tilde b ^r=  - {\textstyle\frac12} \tilde g^{ab} \tgr $,
\smallbreak\qquad
 $\tilde b^r \tilde b^r = {\textstyle\frac14} \tilde g^{ab} \tilde g^{cd}
 \tgr \tilde g_{cd,r}$,
\smallbreak\qquad
 $\tilde b^r \tilde g_{ab} \gtr = {\textstyle\frac12} \tilde g ^{ab} \tilde
 g^{cd} \tgr \tilde g_{cd,r}$,
\smallbreak\qquad
 $\tilde b^r_{,r} = {\textstyle\frac12} \tilde g^{ac} \tilde g^{bd} \tilde
 g _{cd,r} \tilde g _{ab,r} - {\textstyle\frac12} \tilde g^{ab} \tgrr$,
\smallbreak\qquad
 $\tilde g_{ab} \gtrr = 2 \tilde g^{ac} \tilde g^{bd} \tgr \tilde
 g_{cd,r} - \tilde g^{ab} \tgrr$,
\smallbreak\qquad
$ \gtr \tilde g^{cd} _{,r} \left( \tilde g_{ab} \tilde g_{cd} +
 \tilde g_{ac} \tilde g_{bd} + \tilde g_{ad} \tilde g_{bc} \right)=
  \tilde g^{ab} \tilde g^{cd} \tgr \tilde g_{cd,r} + 2 \tilde
 g^{ab} \tilde g^{cd} \tilde g _{ac,r} \tilde g_{bd,r}$.
\medbreak\noindent
Performing the integrations we obtain, modulo normalizing constants of $\pi^{(m-1)/2}
 \sqrt { \tilde g}$ one obtains:
\medbreak\quad
$\tilde \alpha_I= \X \left[ - \frac{ \sqrt{-1}} {16} d_{1051} + \frac{ \sqrt{-1}} 8 d_{1031}
+ \sqrt{-1} d_{3013} - \frac{ \sqrt{-1}} 2 d_{1013} \right.$
\smallbreak\qquad\qquad
$\left. +2 \sqrt{-1} d_{3014} - \frac{\sqrt{-1}} 2 d_{1012} - \sqrt{-1} d_{1013} \right]$
\smallbreak\qquad
$+ \Y \left[ - \frac{ 7\sqrt{-1}} {64} d_{1071} + \frac 1 4 d_{2043} - \frac 1 8 d_{0042} + \frac{ \sqrt{-1}} 8 d_{1051} \right.$
\smallbreak\qquad\qquad
$ - \frac{ \sqrt{-1}} 8 d_{1031} - \sqrt{-1} d_{3013} + \sqrt{-1} d_{1013} - 4 \sqrt{-1} d_{3014}$
\smallbreak\qquad\qquad
$\left. - \sqrt{-1} d_{1014} + 6\sqrt{-1} d_{3015} + \frac{ \sqrt{-1}} 2 d_{1012} + 2 \sqrt{-1} d_{1013} - 3 \sqrt{-1} d_{1014} \right]$
\smallbreak\qquad
$+ \Z \left[ - \frac{ 7 \sqrt{-1}}{128} d_{1071} + \frac 1 8 d_{2043} - \frac 1 {16} d_{0042}
 + \frac{ \sqrt{-1}} {16} d_{1051} + \frac 1 {16} d_{2042} \right.$
\smallbreak\qquad\qquad
$ - \frac 1 4 d_{2023}  + \frac 1 8 d_{0022} - \frac{ \sqrt{-1}} {32} d_{1031} - \frac 1 8 d_{2022}
+ \frac {\sqrt{-1}} 4 d_{3013}- \frac{\sqrt{-1}} 4 d_{1013}$
\smallbreak\qquad\qquad
$ \left. + \frac{ 3\sqrt{-1}} 2 d_{3014}  - \frac{ \sqrt{-1}}
2 d_{1014}+ 3 \sqrt{-1} d_{3015} - \frac {\sqrt{-1}} 2 d_{1013} - \frac{ 3 \sqrt{-1}} 2 d_{1014} \right]$
\smallbreak\qquad
$+ SL_{aa} \left[ \frac{ \sqrt{-1}} 4 d_{1051} + \frac 1 4 d_{0041} - \frac{ \sqrt{-1}} 2 d_{1031}
 - \frac 1 2 d_{0021} + \frac{\sqrt{-1}} 4 d_{1051} - \frac{ \sqrt{-1}} 2 d_{1031}\right.$
\smallbreak\qquad\qquad
$ \left.- d_{2022} - 2 d_{2023}  + d_{0022} + \sqrt{-1} d_{1012} + 2 \sqrt{-1} d_{1013} \right]$
\smallbreak\qquad
$ + S^2 \left[ - \sqrt{-1} d_{1031} - d_{0021} \right],$
\medbreak\quad
$\beta_I = \X \left[ - \frac{ \sqrt{-1}} {16} d_{1141} + \frac{\sqrt{-1}} 8 d_{1121} \right] $
\smallbreak\qquad
$ + \Y \left[ - \frac{ 7 \sqrt{-1}} {64} d_{1161} + \frac 1 4 d_{2133} - \frac 1 8 d_{0132} + \frac {\sqrt{-1}} 8 d_{1141}
    - \frac {\sqrt{-1}} 8 d_{1121} \right] $
\smallbreak\qquad
$+ \Z \left[ - \frac{ 7 \sqrt{-1}} {128} d_{1161} + \frac 1 8 d_{2133} - \frac 1 {16} d_{0132}
+ \frac{\sqrt{-1}} {16} d_{1141} + \frac 1 {16} d_{2132}\right.$
\smallbreak\qquad\qquad
$\left.  - \frac 1 4 d_{2113}  + \frac 1 8 d_{0112} - \frac{ \sqrt{-1}} {32} d_{1121} - \frac 1 8 d_{2112}
\right]$
\smallbreak\qquad
$ + SL_{aa} \left[ \frac{ \sqrt{-1}} 4 d_{1141} + \frac 1 4 d_{0131} - \frac{\sqrt{-1}} 2 d_{1121} - \frac 1 2 d_{0111} \right]$,
\medbreak\quad
$\gamma_I = \X \left[ - \frac{\sqrt{-1}}{16} d_{1231} + \frac{\sqrt{-1}} 8 d_{1211} \right] $
\smallbreak\qquad
$+\Y \left[ - \frac{ 5 \sqrt{-1}} {64} d_{1251} + \frac {\sqrt{-1}} 8 d_{1231}
- \frac{ \sqrt{-1}} 8 d_{1211} - \frac {\sqrt{-1}} {32} d_{1251} \right.$
\smallbreak\qquad\qquad
$\left. + \frac 1 4 d_{2223} - \frac 1 8 d_{0222} \right]$
\smallbreak\qquad
$+ \Z \left[ - \frac{ 5 \sqrt{-1}} {128} d_{1251} + \frac{ \sqrt{-1}} {16} d_{1231}
- \frac{ \sqrt{-1}} {64} d_{1251}+ \frac 1 {16} d_{2222}  \right.$
\smallbreak\qquad\qquad
$\left.  + \frac 1 8 d_{2223} - \frac 1 {16} d_{0222}  - \frac {\sqrt{-1}} {32} d_{1211} \right]$
\smallbreak\qquad
$+S L_{aa} \left[ \frac{\sqrt{-1}} 4 d_{1231} + \frac 1 4 d_{0221}\right]$,
\medbreak\quad
$\delta_I=\X \left[ - \frac{\sqrt{-1}} {24} d_{1321}\right]+ \Y \left[ - \frac{5\sqrt{-1}}{96} d_{1341} +
\frac{\sqrt{-1}} {12} d_{1321} \right]$
\smallbreak\qquad
$+\Z \left[ - \frac{5 \sqrt{-1}}{192} d_{1341} + \frac{\sqrt{-1}}{32} d_{1321} \right]$,
\medbreak\quad
$\epsilon_I = \Y \left[ - \frac{\sqrt{-1}} {64} d_{1431} \right] + \Z \left[ - \frac{ \sqrt{-1}} {128} d_{1431} \right]$.
\medbreak\noindent
Adding up all terms and simplifying using the functional equation
and the doubling formula for the $\Gamma$-function, the
contribution to the heat kernel coefficient reads
\medbreak\quad
$(4\pi)^{-m/2} \kappa_{\alpha -2} \sqrt{ \tilde g} \left\{ \frac 8 {(1-\alpha ) (2-\alpha ) } \,\, S^2
+ \frac{ 2 (\alpha ^2 - \alpha -8)} {(\alpha -1) (\alpha -2) (\alpha -4) }  S L_{aa}\right.$
\medbreak\qquad
$
   + \left( \frac 1 {12} - \frac 1 {3 (1-\alpha )} \right) \tilde g^{ab} \tilde g _{ab,rr}
+\frac 1 4 \Y \left( \frac 1 {1-\alpha } - \frac 1 6 - \frac{ \alpha ^3 - 10 \alpha ^2
   + 21 \alpha +4}{4 (\alpha -6) (\alpha -4) (\alpha -1)} \right)$
\medbreak\qquad
$\left. + \frac 1 4 \Z  \left( - \frac 1 {3 (1-\alpha )} + \frac{ \alpha ^4 - 6 \alpha ^3 - \alpha ^2
  -2\alpha + 104}{8 (\alpha -6) (\alpha -4) (\alpha -2) (\alpha -1)} \right)\right\}$.
\medbreak\noindent
This allows us conclude:
\begin{eqnarray*}
\vartheta_\alpha ^4 &=&
\textstyle\frac{ \alpha ^4 - 6 \alpha ^3 - \alpha ^2 -2\alpha + 104}{8 (\alpha -6) (\alpha -4) (\alpha -2) (\alpha
-1)},\nn\\
\vartheta_\alpha ^5 &=&
\textstyle- \frac{ \alpha ^3 - 10 \alpha ^2 + 21 \alpha +4}{4 (\alpha -6) (\alpha -4) (\alpha -1)} ,\nn\\
\vartheta_\alpha ^6 &=&
\textstyle \frac{ 2 (\alpha ^2-\alpha -8)} {(\alpha -1) (\alpha -2) (\alpha -4)}. \nn
\end{eqnarray*}
The values for $\alpha =0$ reproduce the result for the smooth setting. We also confirm the result for $\vartheta_\alpha^7$.
We can now employ Lemma \ref{lem-4.1} (2) to determine $\vartheta_\alpha^8$. We find
$$\textstyle\vartheta_\alpha^8 = - \frac 1 {(\alpha -1) (\alpha -4)},$$ which in the limit $\alpha \to 0$
reproduces the correct answer.

{\bf Acknowledgements:} Research of PG is supported by project MTM2009-07756
(Spain). Research of KK is supported by the National Science Foundation Grant PHY-0757791. Part of
this work was done while KK visited the Max-Planck-Institute for Mathematics in the Sciences and
KK thanks in particular Eberhard Zeidler and J{\"u}rgen Jost for their hospitality.

\bibliographystyle{plain}
%\bibliography{addarchive,addbook,archive,addunpub}

\end{document}